\documentclass[10pt,reqno]{amsart}
\usepackage{amssymb,mathrsfs,graphicx}
\usepackage{ifthen}
\usepackage{hyperref}
\usepackage[margin=1in]{geometry}
\usepackage{caption}
\usepackage{sidecap}
\usepackage{rotating,dsfont}
\usepackage{enumitem}

\usepackage{soul}
\usepackage{cancel}

\usepackage{tabularx}

\usepackage{colortbl}
\definecolor{black}{rgb}{0.0, 0.0, 0.0}
\definecolor{red}{rgb}{1.0, 0.5, 0.5}
\provideboolean{shownotes} 
\setboolean{shownotes}{true} 
\newcommand{\margnote}[1]{
\ifthenelse{\boolean{shownotes}}%
{\marginpar{\raggedright\tiny\texttt{#1}}}%
{}%
}
\newcommand{\hole}[1]{
\ifthenelse{\boolean{shownotes}}%
{\begin{center} \fbox{ \rule {.25cm}{0cm} \rule[-.1cm]{0cm}{.4cm}
\parbox{.85\textwidth}{\begin{center} \texttt{#1}\end{center}} \rule
{.25cm}{0cm}}\end{center}} {} }

\title[Damped Euler--Poisson: large solutions and overdamped limits]{Global large smooth solutions and overdamped limits
for the damped isothermal Euler--Poisson system}

\author[Choi]{Young-Pil Choi}
\address[Young-Pil Choi]{\newline Department of Mathematics\newline
Yonsei University, 50 Yonsei-Ro, Seodaemun-Gu, Seoul 03722, Republic of Korea}
\email{ypchoi@yonsei.ac.kr}

\author[Jung]{Jinwook Jung}
\address[Jinwook Jung]{\newline Department of Mathematics and  Research Institute for Natural Sciences \newline
Hanyang University, 222 Wangsimni-ro, Seongdong-gu, Seoul 04763, Republic of Korea}
\email{jinwookjung@hanyang.ac.kr}

\numberwithin{equation}{section}

\newtheorem{theorem}{Theorem}[section]
\newtheorem{lemma}{Lemma}[section]

\newtheorem{proposition}{Proposition}[section]
\newtheorem{remark}{Remark}[section]

\newcommand{\R}{\mathbb R}

\newcommand{\N}{\mathbb N}

\newcommand{\T}{\mathbb T}

\newcommand{\bq}{\begin{equation}}
\newcommand{\eq}{\end{equation}}

\newcommand{\lt}{\left}
\newcommand{\rt}{\right}

\newcommand{\pa}{\partial}

\newcommand{\intt}{\int_{\T^d}}

\newcommand{\calC}{\mathcal C}

\newcommand{\calE}{\mathcal E}
\newcommand{\calF}{\mathcal F}

\newcommand{\calH}{\mathcal H}
\newcommand{\calI}{\mathcal I}

\newcommand{\calM}{\mathcal M}
\newcommand{\calN}{\mathcal N}

\newcommand{\calR}{\mathcal R}

\makeatletter
\def\moverlay{\mathpalette\mov@rlay}
\def\mov@rlay#1#2{\leavevmode\vtop{%
   \baselineskip\z@skip \lineskiplimit-\maxdimen
   \ialign{\hfil$\m@th#1##$\hfil\cr#2\crcr}}}
\newcommand{\charfusion}[3][\mathord]{
    #1{\ifx#1\mathop\vphantom{#2}\fi
        \mathpalette\mov@rlay{#2\cr#3}
      }
    \ifx#1\mathop\expandafter\displaylimits\fi}
\makeatother

\begin{document}
\allowdisplaybreaks

\date{\today}


\begin{abstract}  
We consider the isothermal Euler--Poisson system with linear damping on a periodic domain in the large damping regime. For arbitrarily large smooth initial data with density bounded away from vacuum, we prove the global-in-time existence of smooth solutions. The argument is based on a large-damping bootstrap scheme, a modified density estimate revealing hidden parabolic dissipation, top-order weighted cancellations, and a comparison with an auxiliary drift-diffusion--Poisson system. We further prove exponential relaxation to the homogeneous equilibrium and establish a large-data overdamped limit as the damping coefficient tends to infinity. In the slow time scale, the density converges quantitatively to the large smooth solution of the drift-diffusion--Poisson system with the same initial density. After subtracting a fast initial layer from the rescaled flux, the flux converges quantitatively to the corresponding drift-diffusion flux.
\end{abstract}

\maketitle \centerline{\date}

\tableofcontents

%
%
%
%

\section{Introduction} 

We consider the isothermal Euler--Poisson system with linear damping on the flat torus $\T^d$:
\begin{align}\label{eq:main}
\begin{aligned}
&\pa_t \rho+\nabla\cdot(\rho u)=0, \quad t>0,\ x\in\T^d,\cr
&\pa_t(\rho u)+\nabla\cdot(\rho u\otimes u) + \nabla\rho+\rho\nabla\phi =-\nu\rho u,\cr
&-\Delta\phi=\rho-1.
\end{aligned}
\end{align}
Here $\rho=\rho(t,x)>0$ is the density, $u=u(t,x)\in\R^d$ is the velocity field, $\phi=\phi(t,x)$ is the electrostatic potential, and $\nu>0$ is the damping strength. The system is supplemented with the initial data
\bq\label{eq:initial}
(\rho,u)|_{t=0}=(\rho_0,u_0).
\eq
Throughout the paper, we normalize the total mass and the Poisson potential by
\[
 \intt \rho_0\,dx=1,\quad  \intt \phi(t,x)\,dx=0.
\]
The first condition is the compatibility condition for the Poisson equation, while the second one fixes the additive constant in $\phi$. We also assume, without loss of generality, that
\[
 \intt \rho_0 u_0\,dx=0.
\]
Indeed,
\[
 \frac{d}{dt}\intt \rho u\,dx  =  -\nu\intt \rho u\,dx,
\]
so this condition removes the homogeneous momentum mode.

The purpose of this paper is to study the large damping regime for \eqref{eq:main}. We allow the initial perturbation from the homogeneous
state to be large in Sobolev norms, while requiring the damping coefficient $\nu$ to be sufficiently large depending on the initial size and the pointwise density bounds. The main difficulty is that the damping acts directly only on the velocity. The density is affected indirectly through the continuity equation and the force balance involving both the pressure and the self-consistent Poisson field. Thus, in the large-data regime, the Poisson force cannot be treated simply as a lower-order perturbation; it has to be incorporated into the dissipative and cancellation structures of the estimates.

Let us briefly review some related results. The local-in-time theory for quasilinear symmetric hyperbolic systems is classical \cite{Kat75,Lax73,Maj84}. On the other hand, smooth solutions of compressible Euler equations may form singularities in finite time for large initial data \cite{BCG25, MRRS22, Sid85}. Linear damping can suppress this mechanism for small perturbations of constant states and leads to global existence and relaxation for damped Euler equations \cite{STW03,WY01}.  More recently, a large-damping theory for the isothermal Euler equations with damping was developed in \cite{Pen24}, where global large smooth solutions and a large-data relaxation limit toward the heat equation were obtained, together with a quantitative flux error estimate after subtracting a fast initial layer. Related large-damping global dynamics for damped isothermal Euler equations with prescribed exterior potentials were studied in \cite{CTZpre}, including large-data global well-posedness, sharp algebraic decay for small perturbations, and blow-up phenomena in the pressureless regime.

The Euler--Poisson system has a richer structure because the density is coupled to a self-consistent potential. This coupling gives rise to delicate threshold phenomena and strongly influences the global dynamics. Critical thresholds and global regularity have been studied for Euler--Poisson systems and related nonlocal-force models, especially in one-dimensional, pressureless, or radially symmetric settings \cite{BL20,BLpre, CCTT16, CCZ16,  CKKT26,ELT01,TW08}. In the undamped case, global smooth solutions near constant equilibria have been obtained in several dimensions under suitable smallness, irrotationality, or structural assumptions \cite{Guo98,GHZ17, IP13,JLZ14,LW14}.  Related Euler--Riesz systems have also been studied in connection with global smooth solutions and damping effects \cite{CJu23, CJL24,CJL25,CJL26, DD22}.

Large-time behavior for damped Euler--Poisson and related hydrodynamic semiconductor models has also been extensively investigated near constant or stationary states. For small perturbations, global existence, asymptotic stability, and decay estimates have been obtained by energy methods, Green function analysis, spectral analysis, and time-weighted energy estimates \cite{AJ03,GHL03, HMW03,LY12,TTX20, WQ15, WW14,XK15}.  These results are mostly perturbative around constant or stationary states, whereas our global existence result allows the initial perturbation to be arbitrarily large in Sobolev norms, provided the damping coefficient is chosen sufficiently large and the initial density is bounded away from vacuum.

Relaxation and high-friction limits form another important direction. For damped Euler equations, strong relaxation limits toward parabolic equations have been established in both isothermal and isentropic regimes \cite{CG07,GH19, HMP05, HPW11,LC13}. For Euler--Poisson, Euler--Riesz, and related nonlocal systems, modulated energy and relative entropy methods have been used to justify high-friction limits toward gradient-flow, porous-medium, or aggregation-diffusion type equations \cite{ACC24, AH26,Cho21,CJe21, LT17}. These works reveal the natural parabolic or gradient-flow structure emerging from Euler-type systems with strong friction.

The present work addresses the large-damping problem in the self-consistent Euler--Poisson setting, together with the associated
large-time relaxation and overdamped limit. First, we prove global-in-time smooth solutions for arbitrarily large smooth initial perturbations with strictly positive density, provided that the damping coefficient $\nu$ is sufficiently large. A key feature of the analysis is that the Poisson force cannot be treated merely as a lower-order perturbation. Instead, it must be incorporated into both the dissipative structure and the high-order cancellations needed to close the large-data estimates.

Second, we justify the corresponding large-data overdamped limit. Compared with the damped isothermal Euler equations, whose large-damping limit is the heat equation, the Poisson coupling changes the limiting dynamics to a drift-diffusion--Poisson system. In the slow time scale $s=t/\nu$, the density converges quantitatively to the large smooth solution of the limiting system, and the rescaled flux, after subtracting a damped-heat initial layer, converges strongly to the corresponding drift-diffusion flux. The same overdamped structure also explains the slow exponential relaxation rate of the original Euler--Poisson solution.

%
%
%
%
%
\subsection{Main results}

We first state the global existence and relaxation result.
\begin{theorem}\label{thm:main}
Let $2\le d\le4$ and let $m\in\N$ satisfy $m>\frac d2$. Assume that $(\rho_0,u_0)\in H^{m+1}(\T^d)\times H^{m+1}(\T^d)$ and that
\[
\rho_1\le \rho_0(x)\le \rho_2,\quad x\in\T^d,
\]
for some constants $\rho_1,\rho_2>0$. Then there exists
\[
\nu_0=\nu_0(d,m,\rho_1,\rho_2,\|\rho_0-1\|_{H^{m+1}},\|u_0\|_{H^{m+1}})>0
\]
such that, for every $\nu\ge\nu_0$, the isothermal Euler--Poisson system \eqref{eq:main}--\eqref{eq:initial} admits a global-in-time solution
\[
(\rho,u)\in C([0,\infty);H^{m+1}(\T^d)) \times C([0,\infty);H^{m+1}(\T^d)).
\]
Moreover, the density remains uniformly bounded away from vacuum.

Furthermore, the solution converges exponentially to the homogeneous equilibrium. More precisely, there exist constants $C>0$ and $c>0$, independent of $t$ and $\nu$, such that
\[
\|\rho(t)-1\|_{H^{m+1}} + \|u(t)\|_{H^{m+1}} \le C e^{-\theta t}, \quad t\ge0,
\]
where
\[
\theta:=c\min\lt\{\nu,\frac1\nu\rt\}.
\]
In particular, since $\nu_0$ may be chosen so that $\nu_0\ge1$, one has $\theta=\frac{c}\nu$ in the large damping regime considered here.
\end{theorem}

\begin{remark} 
The large damping assumption in Theorem \ref{thm:main} should be understood as a quantitative perturbative condition for the nonlinear energy estimates. In the proof, the bootstrap quantity controls the $H^{m+1}$ size of $(\rho-1,u)$ together with the time-integrated velocity dissipation, and the a priori estimates are closed when this quantity is sufficiently small relative to the damping strength $\nu$.

Thus, the theorem allows arbitrary fixed smooth initial data with positive density, provided $\nu$ is chosen sufficiently large depending on the initial size and the density bounds. For fixed $\nu$, the same mechanism would yield a small-data result. We state the result in the large damping form because this is the regime in which the argument gives large-data global existence.
\end{remark}

\begin{remark} 
For $\nu\ge1$, the decay rate in Theorem \ref{thm:main} is of order $\nu^{-1}$. This is consistent with the overdamped dynamics. Indeed, in the large damping regime, the momentum equation formally reduces to the force balance
\[
 \nu\rho u\simeq -\nabla\rho-\rho\nabla\phi,
\]
or equivalently
\[
 u\simeq -\frac1\nu (\nabla\log\rho+\nabla\phi).
\]
Substituting this relation into the continuity equation gives
\[
 \partial_t\rho  \simeq  \frac1\nu\nabla\cdot(\nabla\rho+\rho\nabla\phi).
\]
Thus, the density evolves on the slow parabolic time scale
\[
 s=\frac{t}{\nu}.
\]
An order-one exponential relaxation in the slow time variable corresponds to the rate $e^{-ct/\nu}$ in the original time variable. This explains why the natural decay rate in the large damping regime is $\theta=c/\nu$, even though the velocity itself is directly damped at rate $\nu$.
\end{remark}

\begin{remark} 
The potential decay is an immediate consequence of the density decay. Since
\[
-\Delta\phi=\rho-1, \quad \intt \phi\,dx=0,
\]
by elliptic regularity we have
\[
\|\nabla\phi(t)\|_{H^{m+2}} \le C\|\rho(t)-1\|_{H^{m+1}}.
\]
Thus Theorem \ref{thm:main} also yields
\[
\|\nabla\phi(t)\|_{H^{m+2}} \le C e^{-\theta t}.
\]
\end{remark}

We now turn to the overdamped limit associated with the large damping regime. We work in the slow time variable
\[
s=\frac{t}{\nu}.
\]
For the global solution $(\rho^\nu,u^\nu,\phi^\nu)$ constructed in Theorem \ref{thm:main}, we define the slow-time density and potential by
\[
\rho_\nu(s,x):=\rho^\nu(\nu s,x), \quad \Phi_\nu(s,x):=\phi^\nu(\nu s,x),
\]
and the rescaled flux by
\[
J_\nu(s,x):=\nu\rho^\nu u^\nu(\nu s,x).
\]
Then $(\rho_\nu,J_\nu,\Phi_\nu)$ satisfies
\[
 \pa_s\rho_\nu+\nabla\cdot J_\nu=0, \quad -\Delta\Phi_\nu=\rho_\nu-1,
\]
and
\[
\frac1{\nu^2} \pa_sJ_\nu + \frac1{\nu^2}\nabla\cdot \lt( \frac{J_\nu\otimes J_\nu}{\rho_\nu} \rt) + \nabla\rho_\nu + \rho_\nu\nabla\Phi_\nu = -J_\nu.
\]
Formally, as $\nu\to\infty$, the flux is constrained by the drift-diffusion flux relation
\[
\bar J=-\nabla\bar\rho-\bar\rho\nabla\bar\Phi,
\]
which can also be viewed as a Darcy-type law with the self-consistent Poisson drift included. The limiting density solves
\[
 \pa_s\bar\rho+\nabla\cdot\bar J=0, \quad -\Delta\bar\Phi=\bar\rho-1,
\]
or equivalently
\[
 \pa_s\bar\rho -\nabla\cdot(\bar\rho\nabla\bar\Phi) = \Delta\bar\rho.
\]

Since the rescaled initial flux
\[
 J_\nu(0,x)=\nu\rho_0(x)u_0(x)
\]
is generally of order $O(\nu)$, the flux convergence will be stated after subtracting a fast initial layer.

\begin{theorem} \label{thm:main2}
Let the assumptions of Theorem \ref{thm:main} hold. Let $(\bar\rho,\bar J,\bar\Phi)$ be the unique smooth solution of
\[
 \pa_s\bar\rho+\nabla\cdot\bar J=0, \quad \bar J=-\nabla\bar\rho-\bar\rho\nabla\bar\Phi, \quad -\Delta\bar\Phi=\bar\rho-1, \quad \bar\rho(0)=\rho_0.
\]
Then, as $\nu\to\infty$, the slow-time density satisfies
\[
\rho_\nu\to\bar\rho \quad\text{strongly in}\quad L^\infty(0,\infty;H^m(\T^d)) \cap L^2(0,\infty;H^{m+1}(\T^d)).
\]
More precisely, there exists a constant $C>0$, independent of $\nu$, such
that
\bq\label{eq:od-den-err}
\sup_{s\ge0} \|\rho_\nu(s)-\bar\rho(s)\|_{H^m}^2 + \int_0^\infty \|\rho_\nu(s)-\bar\rho(s)\|_{H^{m+1}}^2\,ds \le \frac{C}{\nu^2}.
\eq
Moreover, let $J_L^\nu$ be the unique solution of the damped heat equation:
\[
\frac1{\nu^2} \lt( \pa_sJ_L^\nu-\Delta J_L^\nu \rt) + J_L^\nu = 0, \quad J_L^\nu(0,x)=\nu\rho_0(x)u_0(x).
\]
Then, after subtracting this initial layer, the rescaled flux converges to the drift-diffusion flux $\bar J$:
\[
J_\nu-J_L^\nu \to \bar J \quad\text{strongly in}\quad L^2(0,\infty;H^{m-1}(\T^d)).
\]
More precisely,
\bq\label{eq:od-flux-err}
\int_0^\infty \|J_\nu(s)-J_L^\nu(s)-\bar J(s)\|_{H^{m-1}}^2\,ds \le \frac{C}{\nu^2}.
\eq
\end{theorem}

\begin{remark} 
The role of $J_L^\nu$ is to remove the large rescaled initial flux $J_\nu(0)=\nu\rho_0u_0$. Although the dominant initial relaxation is caused
by the fast damping, we define $J_L^\nu$ by a damped heat equation rather than by a purely damped ODE. This choice is suited to the Sobolev flux estimate, since the heat part is compatible with the parabolic structure of the overdamped limit and provides the regularization needed to control the initial layer in $H^{m-1}$. After subtracting this layer, the remaining flux satisfies a uniform defect estimate and converges strongly to the drift-diffusion flux.
\end{remark}

\begin{remark} 
The limiting drift-diffusion--Poisson system in Theorem \ref{thm:main2} is globally well posed in the smooth class considered here. This follows
from the analysis of the auxiliary drift-diffusion system in Section \ref{sec:aux-comp} below. After the slow-time change of variables $s=t/\nu$, that auxiliary system becomes exactly the limiting equation. Moreover, the maximum principle gives
\[
 \rho_1\le \bar\rho(s,x)\le \rho_2,\quad s\ge0,
\]
and the high-order Sobolev estimates then provide global smoothness.
\end{remark}

\begin{remark}
The overdamped limit can also be viewed from the perspective of modulated energy methods for large-friction limits \cite{ACC24,Cho21,CJe21,LT17}. In such an approach, one compares the Euler--Poisson solution with the drift-diffusion--Poisson solution by a relative free energy consisting of the isothermal internal energy and the Poisson interaction energy, together with a modulated kinetic energy. This naturally yields a low-order relaxation estimate and highlights the gradient-flow structure of the limiting equation.

In the present work, however, we use the auxiliary density comparison and a flux defect estimate. This gives a direct high-order density error estimate and, after subtracting the fast initial layer, a strong flux error estimate in Sobolev norms.
\end{remark}

%
%
%
%
%
\subsection{Ideas of the proof}
We briefly describe the main ingredients of the proof. Setting
\[
h:=\rho-1,
\]
we rewrite \eqref{eq:main} as
\begin{align}\label{eq:h-system}
\begin{aligned}
&\pa_t h+\nabla\cdot(hu)+\nabla\cdot u=0,\cr
&\pa_t u+u\cdot\nabla u+\frac{\nabla h}{1+h}+\nabla\phi=-\nu u,\cr
&-\Delta\phi=h.
\end{aligned}
\end{align}
The damping term directly controls the velocity, but not the density. Thus, one has to recover density dissipation through the coupling between the continuity equation, the pressure, and the Poisson force. Moreover, since the density fluctuation $\rho_0-1$ is not assumed to be small
in $L^\infty$, uniform pointwise bounds for $\rho$ cannot be closed by a purely perturbative argument around the constant state.

A key ingredient is the auxiliary drift-diffusion system
\[
\pa_t\tilde\rho -\frac1\nu\nabla\cdot(\tilde\rho\nabla\tilde\phi) = \frac1\nu\Delta\tilde\rho, \quad -\Delta\tilde\phi=\tilde\rho-1,
\]
with initial data 
\[
\tilde\rho(0,x)=\rho_0(x). 
\]
This system represents the parabolic dynamics formally obtained in the large damping regime. It also satisfies a maximum principle, which preserves the lower and upper bounds of the initial density. Thus, the auxiliary solution provides a natural reference density with good positivity properties.

The comparison between the Euler--Poisson density and the auxiliary density is not made directly through $h-\tilde h$, where $\tilde h:=\tilde\rho-1$. Instead, the correct comparison variable is
\[
h-\tilde h-\frac1\nu\nabla\cdot(\rho u).
\]
This correction is dictated by the momentum equation and reflects the parabolic structure hidden in the damped Euler--Poisson system. Indeed, after rewriting the density equation by means of the momentum equation, one obtains a drift-diffusion type equation for $h$ with an additional error involving $\pa_t(\rho u)$ and $\nabla\cdot(\rho u\otimes u)$. The correction term $\nu^{-1}\nabla\cdot(\rho u)$ removes this time derivative and leads to a dissipative estimate for the corrected difference, which in turn controls $h-\tilde h$.

The second main ingredient is a high-order hyperbolic energy estimate for \eqref{eq:h-system}. Since the isothermal pressure force is
\[
 \frac{\nabla h}{1+h},
\]
the top-order pressure term cannot be treated as a lower-order error. The main cancellation is obtained by testing the differentiated continuity
equation with
\[
 (1+h)^{-2}\partial^\alpha h.
\]
This produces a leading term that cancels exactly with the corresponding top-order pressure term in the differentiated velocity equation. Together with a related weighted estimate controlling the Poisson contribution at one lower derivative level, this weighted cancellation closes the $H^{m+1}$ estimate without loss of derivatives.

The exponential relaxation is obtained through a compensated energy argument. Although the velocity is directly damped, the density and the
Poisson field have no direct dissipation at the basic energy level. We add suitable cross terms involving the velocity, the density, and the Poisson potential, in the spirit of hypocoercivity estimates for damped Euler--Poisson dynamics. These terms transfer the velocity damping to the density and potential variables and yield a closed low-order Lyapunov inequality.  The resulting rate is of order $\nu^{-1}$ in the large damping regime. Combining the low-order decay with the uniform $H^{m+1}$ bound, we obtain decay estimates for intermediate Sobolev norms and the $W^{1,\infty}$ norms by interpolation. These estimates are then used in the top-order energy argument to derive exponential decay in $H^{m+1}$.

Finally, the overdamped limit is proved by combining the density comparison estimate with a flux defect estimate. In the slow time variable $s=t/\nu$, the auxiliary drift-diffusion density becomes exactly the solution of the limiting drift-diffusion--Poisson system. Thus, the comparison estimate between the Euler--Poisson density and the auxiliary density yields the quantitative convergence of $\rho_\nu$ to $\bar\rho$ in the Sobolev norms as stated in Theorem \ref{thm:main2}. The convergence of the flux requires an additional argument. The rescaled flux
\[
 J_\nu=\nu\rho^\nu u^\nu(\nu s)
\]
satisfies a relaxation equation whose leading-order constraint is the Darcy-type relation
\[
 J_\nu \simeq -\nabla\rho_\nu-\rho_\nu\nabla\Phi_\nu.
\]
However, the initial value $J_\nu(0)=\nu\rho_0u_0$ is in general of order $O(\nu)$ and creates a fast initial layer. We therefore subtract the
damped-heat layer $J_L^\nu$ and estimate the remaining flux defect. This yields strong convergence of $J_\nu-J_L^\nu$ to the drift-diffusion flux $\bar J$ in $L^2(0,\infty;H^{m-1}(\T^d))$.

%
%
%
%
%
\subsection{Organization of the paper}

The paper is organized as follows. Section \ref{sec:pre} collects the local well-posedness result, the basic energy estimate, and the functional inequalities used throughout the paper. In Section \ref{sec:apri}, we derive the main a priori estimates: a modified energy estimate for the density and a top-order hyperbolic estimate based on weighted cancellations. Section \ref{sec:aux-comp} introduces the auxiliary drift-diffusion density and proves the comparison estimate needed to close the pointwise density bounds. In Section \ref{sec:proof-main}, we close the bootstrap argument, prove global existence, and establish the exponential relaxation asserted in Theorem \ref{thm:main}. Finally, Section \ref{sec:overdamped-limit} proves the overdamped limit and the flux error estimate stated in Theorem \ref{thm:main2}.

%
%
%
%
%

\section{Preliminaries}\label{sec:pre}

In this section, we collect the standard well-posedness framework, the basic energy identity, and two functional inequalities used in the a priori estimates below.

%
%
%
%
%
\subsection{Local well-posedness and free energy}

We first recall the local theory and the free energy estimate for smooth solutions with strictly positive density.

\begin{theorem}\label{thm:lwp}
Let $m>d/2$ and assume that $(\rho_0,u_0)\in H^{m+1}(\T^d)\times H^{m+1}(\T^d)$, and
\[
\inf_{x\in\T^d}\rho_0(x)>0.
\]
Then there exists $T_*>0$ such that the isothermal Euler--Poisson system \eqref{eq:main}--\eqref{eq:initial} admits a unique smooth solution $(\rho,u)$ on $[0,T_*]$ satisfying
\[
\rho\in C([0,T_*];H^{m+1}(\T^d)),\quad u\in C([0,T_*];H^{m+1}(\T^d)),
\]
and
\[
\inf_{(t,x)\in[0,T_*]\times\T^d}\rho(t,x)>0.
\]
Moreover, the solution can be continued beyond $T_*$ as long as
\[
\sup_{0\le t<T_*} \lt( \|\rho(t)-1\|_{H^{m+1}}+\|u(t)\|_{H^{m+1}} \rt)<\infty
\]
and
\[
\inf_{0\le t<T_*}\inf_{x\in\T^d}\rho(t,x)>0.
\]
\end{theorem}

This is the standard continuation criterion for symmetric hyperbolic systems with strictly positive density, together with elliptic regularity for the Poisson equation; we omit the proof. We refer to \cite{CJe22, Guo98, IP13, JLZ14, Kat75, Lax73, Maj84} for related local well-posedness results.

We next provide the free energy estimate, which gives the zeroth-order velocity control used later with the top-order estimate.

\begin{lemma} \label{lem:be}
Let $(\rho,u,\phi)$ be a smooth solution to \eqref{eq:main}--\eqref{eq:initial} on $[0,T]$ satisfying
\[
0<\underline\rho\le \rho(t,x)\le \overline\rho<\infty.
\]
Then we have
\[
\sup_{0\le t\le T}\|u(t)\|_{L^2}^2 + \nu\int_0^T\|u(t)\|_{L^2}^2\,dt \le C \lt( \intt \rho_0|u_0|^2\,dx + \intt (\rho_0\log\rho_0-\rho_0+1)\,dx + \|\nabla\phi_0\|_{L^2}^2 \rt),
\]
where $C>0$ depends only on $\underline\rho$ and $\overline\rho$.
\end{lemma}

\begin{proof}
Define
\[
\calE(t) := \frac12\intt \rho |u|^2\,dx + \intt (\rho\log\rho-\rho+1)\,dx + \frac12\intt |\nabla\phi|^2\,dx.
\]
The usual energy identity gives
\[
\frac{d}{dt}\calE(t) + \nu\intt \rho |u|^2\,dx =0.
\]
Integrating in time and using the lower and upper bounds on $\rho$, we obtain the desired estimate.
\end{proof}

%
%
%
%
%
\subsection{Functional inequalities}

We collect two functional inequalities used repeatedly below.  

We first recall the following homogeneous Moser-type composition estimate, see \cite{KM81, Maj84}.

\begin{lemma} \label{lem:pow-est} 
Let $r>d/2$ be an integer, and let $h\in H^r(\T^d)$ satisfy
\[
0<\underline\rho\le 1+h(x)\le \overline\rho, \quad x\in\T^d.
\]
Then, for any $\ell\in\N$, we have
\[
\|\nabla^r(1+h)^{-\ell}\|_{L^2} \le C \lt(1+\|h\|_{H^r}\rt)^{r-1} \|\nabla^r h\|_{L^2}.
\]
Here $C>0$ depends only on $d,r,\ell$ and $\underline\rho$.
\end{lemma}

The next lemma concerns the nonlinear Poisson drift term. It isolates a commutator structure that appears both in the modified density estimate for the Euler--Poisson system and in the Sobolev estimates for the auxiliary drift-diffusion equation.

\begin{lemma} \label{lem:po-drift}
Let $2\le d\le4$ and let $k\ge1$ be an integer. Let $f\in H^{k+1}(\T^d)$ satisfy
\[
\intt  f\,dx=0,
\]
and let $\Phi$ be the zero-mean solution of
\[
-\Delta\Phi=f \quad\text{on }\T^d.
\]
Then, for any multi-index $\alpha$ with $|\alpha|=k$,
\bq\label{eq:p-drift}
\lt| \intt  \pa^\alpha f\, \pa^\alpha\nabla\cdot(f\nabla\Phi)\,dx\rt| \le C \|\nabla^k f\|_{L^2}^2 \|\nabla^{k+1}f\|_{L^2}.
\eq
Here $C>0$ depends only on $d$ and $k$.
\end{lemma}
\begin{proof}
We write
\[
\nabla\cdot(f\nabla\Phi) = \nabla f\cdot\nabla\Phi+f\Delta\Phi.
\]
Thus, we get
\[
 \pa^\alpha\nabla\cdot(f\nabla\Phi)  =  \pa^\alpha\nabla f\cdot\nabla\Phi + f \pa^\alpha\Delta\Phi + \calC_\alpha,
\]
where
\[
\calC_\alpha := \lt[ \pa^\alpha(\nabla f\cdot\nabla\Phi) -  \pa^\alpha\nabla f\cdot\nabla\Phi \rt] + \lt[  \pa^\alpha(f\Delta\Phi) - f  \pa^\alpha\Delta\Phi \rt].
\]
Using $-\Delta\Phi=f$, the two leading terms give
\[
\intt \pa^\alpha f \lt( \pa^\alpha\nabla f\cdot\nabla\Phi + f  \pa^\alpha\Delta\Phi \rt)dx = -\frac12 \intt  f| \pa^\alpha f|^2\,dx.
\]
By H\"older's inequality, we find
\[
\lt| \intt  f| \pa^\alpha f|^2\,dx \rt| \le \|f\|_{L^d} \| \pa^\alpha f\|_{L^{\frac{2d}{d-1}}}^2.
\]
Using the Gagliardo--Nirenberg inequality
\[
\|g\|_{L^{\frac{2d}{d-1}}} \le C \|\nabla g\|_{L^2}^{1/2} \|g\|_{L^2}^{1/2}, \quad 2\le d\le4,
\]
and the Sobolev--Poincar\'e inequality
\[
\|f\|_{L^d} \le C\|f\|_{\dot H^{\frac d2-1}} \le C\|\nabla f\|_{L^2} \le C\|\nabla^k f\|_{L^2},
\]
we obtain
\[
\lt| \intt  f| \pa^\alpha f|^2\,dx \rt| \le C \|\nabla^k f\|_{L^2}^2 \|\nabla^{k+1}f\|_{L^2}.
\]

It remains to estimate the commutator term. By Leibniz rule, every term in $\calC_\alpha$ contains at least one derivative falling on $f$ and at least one derivative falling on the
Poisson factor. More precisely, after writing the derivatives of $\Delta\Phi$ as derivatives of $\Phi$, all terms are of the form
\[
\nabla^p f\,\nabla^{k+2-p}\Phi, \quad 1\le p\le k.
\]
This implies
\[
\lt| \intt \pa^\alpha f\,\calC_\alpha\,dx \rt| \le C \sum_{p=1}^k \| \pa^\alpha f\|_{L^{\frac{2d}{d-1}}} \|\nabla^p f\|_{L^{\frac{2d}{d-1}}} \|\nabla^{k+2-p}\Phi\|_{L^d}.
\]
By the Gagliardo--Nirenberg inequality and elliptic regularity,
\[
\| \pa^\alpha f\|_{L^{\frac{2d}{d-1}}} \le C \|\nabla^{k+1}f\|_{L^2}^{1/2} \|\nabla^k f\|_{L^2}^{1/2}, \quad  \|\nabla^p f\|_{L^{\frac{2d}{d-1}}} \le C \|\nabla^{p+1}f\|_{L^2}^{1/2} \|\nabla^p f\|_{L^2}^{1/2},
\]
and
\[
\|\nabla^{k+2-p}\Phi\|_{L^d} \le C \|\nabla^{k+3-p}\Phi\|_{L^2} \le C \|\nabla^{k+1-p}f\|_{L^2}.
\]
Interpolating between $\|\nabla^k f\|_{L^2}$ and $\|\nabla^{k+1}f\|_{L^2}$, together with Poincar\'e's inequality, gives
\[
\|\nabla^{p+1}f\|_{L^2}^{1/2} \|\nabla^p f\|_{L^2}^{1/2} \|\nabla^{k+1-p}f\|_{L^2} \le C \|\nabla^k f\|_{L^2}^{3/2} \|\nabla^{k+1}f\|_{L^2}^{1/2}.
\]
Hence, we have
\[
\lt| \intt \pa^\alpha f\,\calC_\alpha\,dx \rt| \le C \|\nabla^k f\|_{L^2}^2 \|\nabla^{k+1}f\|_{L^2}.
\]
Combining the estimate of the leading contribution with the above commutator estimate yields \eqref{eq:p-drift}. This completes the proof.
\end{proof}

%
%
%
%
%

\section{A priori estimates}\label{sec:apri}

In this section, we derive the a priori estimates for smooth solutions to the Euler--Poisson system in the large damping regime. We write
\[
\rho=1+h,
\]
so that the system takes the form
\begin{align*}
&\pa_t h+\nabla\cdot((1+h)u)=0,\cr
&\pa_t u+u\cdot\nabla u+\frac{\nabla h}{1+h}+\nabla\phi=-\nu u,\cr
&-\Delta\phi=h.
\end{align*}

Throughout this section, we assume that a smooth solution exists on $[0,T]$ and satisfies the density bootstrap bound
\bq\label{eq:den-boot}
\frac{\rho_1}{2}\le 1+h(t,x)\le 2\rho_2, \quad (t,x)\in[0,T]\times\T^d.
\eq
We also use the bootstrap quantity
\bq\label{eq:MT}
\calM_T := \sup_{0\le t\le T} \lt( \|h(t)\|_{H^{m+1}}+\|u(t)\|_{H^{m+1}} \rt) + \lt( \nu\int_0^T\|u(t)\|_{H^{m+1}}^2\,dt \rt)^{\frac12 }.
\eq
The large damping bootstrap condition is
\bq\label{eq:MTs}
\frac{\calM_T(1+\calM_T)}{\nu}\le \delta_0,
\eq
where $\delta_0>0$ will be chosen sufficiently small.

The estimates below have two different roles. First, we derive a modified energy estimate for the density, which reveals the parabolic dissipation hidden in the damped Euler--Poisson system. Second, we derive a top-order hyperbolic estimate based on weighted cancellations between the continuity equation, the pressure term, and the Poisson force.

%
%
%
%
%

\subsection{Modified energy estimate for the density}

The density equation does not contain direct damping. To recover density dissipation, we use a modified density variable
\bq\label{eq:barh}
\bar h := h-\frac1\nu\nabla\cdot((1+h)u).
\eq
This modification removes the time derivative of $(1+h)u$ which appears when the density equation is rewritten through the momentum equation. The resulting estimate provides the lower-order density control needed for the top-order energy estimate.
 
We now use the modified variable $\bar h$ to derive the density estimate. The nonlinear Poisson drift term will be controlled by Lemma \ref{lem:po-drift}. The proof consists of a zeroth-order estimate followed by an induction on the order of Sobolev norms.

\begin{proposition} \label{prop:mod-den}
Let $2\le d\le4$ and let $m\in\N$ satisfy $m>d/2$. Under the assumptions \eqref{eq:den-boot} and \eqref{eq:MTs}, there exists a constant $C_1>0$, depending
only on $d,m,\rho_1,\rho_2,\|h_0\|_{H^{m+1}}$ and $\|u_0\|_{H^{m+1}}$, but independent of $T$ and $\nu$, such that
\[
\sup_{0\le t\le T}\|h(t)\|_{H^m}^2 + \frac1\nu\int_0^T \|h(t)\|_{H^{m+1}}^2\,dt \le C_1.
\]
\end{proposition}
 
\begin{proof}
We first rewrite the density equation by using the momentum equation. Since
\[
\nu(1+h)u = -\pa_t((1+h)u) -\nabla\cdot((1+h)u\otimes u) -\nabla h -(1+h)\nabla\phi,
\]
taking divergence and using
\[
\pa_t h=-\nabla\cdot((1+h)u)
\]
gives
\[
\pa_t h - \frac1\nu\nabla\cdot((1+h)\nabla\phi) = \frac1\nu\Delta h + \frac1\nu\nabla\cdot \lt[ \pa_t((1+h)u) + \nabla\cdot((1+h)u\otimes u) \rt].
\]
Thus, with $\bar h$ defined by \eqref{eq:barh}, we obtain
\bq\label{eq:barh-eq}
\pa_t\bar h = \frac1\nu\nabla\cdot((1+h)\nabla\phi) + \frac1\nu\Delta h + \frac1\nu\nabla\otimes\nabla:((1+h)u\otimes u).
\eq
Since $-\Delta\phi=h$, we have
\[
\nabla\cdot((1+h)\nabla\phi) = -h+\nabla\cdot(h\nabla\phi).
\]
We also set
\[
R:=\nabla\cdot((1+h)u).
\]

\medskip

\noindent{\bf Step 1. Zeroth-order estimate.} We first prove the estimate at order zero. Since
\[
\bar h=h-\frac1\nu R,
\]
we compute, using \eqref{eq:barh-eq},
\begin{align*}
\frac12\frac{d}{dt}\|\bar h\|_{L^2}^2
&= \frac1\nu \intt h \lt[ \nabla\cdot((1+h)\nabla\phi) + \Delta h + \nabla\otimes\nabla:((1+h)u\otimes u) \rt]dx \cr
&\quad -\frac1{\nu^2} \intt R\lt[ \nabla\cdot((1+h)\nabla\phi) + \Delta h + \nabla\otimes\nabla:((1+h)u\otimes u) \rt]dx .
\end{align*}
For the second term, the momentum equation gives
\[
\nabla\cdot((1+h)\nabla\phi) + \Delta h + \nabla\otimes\nabla:((1+h)u\otimes u) = -\pa_t R-\nu R.
\]
Thus,
\[
-\frac1{\nu^2} \intt R\lt[ \nabla\cdot((1+h)\nabla\phi) + \Delta h + \nabla\otimes\nabla:((1+h)u\otimes u) \rt]dx = \frac12\frac{d}{dt} \lt\|\frac1\nu R\rt\|_{L^2}^2 + \frac1\nu\|R\|_{L^2}^2.
\]
Hence, the modified zero-order energy
\[
\calE_0^h := \|h\|_{L^2}^2 - \frac2\nu \intt h\,R\,dx
\]
satisfies
\bq\label{eq:E0-id}
\frac12\frac{d}{dt}\calE_0^h + \frac1\nu\|h\|_{L^2}^2 + \frac1\nu\|\nabla h\|_{L^2}^2 = {\mathsf I}_1^0 + {\mathsf I}_2^0 + {\mathsf I}_3^0,
\eq
where
\[
{\mathsf I}_1^0 := \frac1\nu \intt h\,\nabla\cdot(h\nabla\phi)\,dx, \quad {\mathsf I}_2^0 := \frac1\nu \intt h\,\nabla\otimes\nabla:((1+h)u\otimes u)\,dx,
\]
and
\[
{\mathsf I}_3^0 := \frac1\nu\|R\|_{L^2}^2.
\]

We first treat ${\mathsf I}_1^0$. Since $-\Delta\phi=h$, integration by parts gives
\[
{\mathsf I}_1^0 = -\frac1{2\nu} \intt h^3\,dx.
\]
This gives
\[
\frac1\nu\|h\|_{L^2}^2 - {\mathsf I}_1^0 = \frac1\nu \intt \lt(1+\frac h2\rt)h^2\,dx.
\]
Since $1+h=\rho\ge \rho_1/2$ under the density bootstrap,
\[
1+\frac h2=\frac{1+\rho}{2}\ge \frac12.
\]
Hence, we obtain
\bq\label{eq:I10}
\frac1\nu\|h\|_{L^2}^2 - {\mathsf I}_1^0 \ge \frac1{2\nu}\|h\|_{L^2}^2.
\eq

We next estimate ${\mathsf I}_2^0 + {\mathsf I}_3^0$. Integrating by parts once,
\[
{\mathsf I}_2^0 = -\frac1\nu \sum_{i,j=1}^d \intt   \pa_i h\,  \pa_j\lt((1+h)u_i u_j\rt)\,dx = -\frac1\nu \intt  (u\cdot\nabla h)^2\,dx +\calR_0^u,
\]
where
\[
\calR_0^u := -\frac1\nu \sum_{i,j=1}^d \intt   \pa_i h \lt[  \pa_j\lt((1+h)u_i u_j\rt) -u_i u_j \pa_jh \rt]dx =  -\frac1\nu \sum_{i,j=1}^d \intt (1+h) \pa_i h  \pa_j (u_i u_j) dx.
\]
By the product estimate, Sobolev embedding, and Young's inequality,
\bq\label{eq:R0u}
|\calR_0^u| \le \frac{C}{\nu}\|\nabla h\|_{L^2}\|1+h\|_{L^\infty}\|\nabla (u \otimes u)\|_{L^2} \le \frac{c}{4\nu}\|\nabla h\|_{L^2}^2 + \frac{C(1+\calM_T)^2}{\nu}\|u\|_{H^{m+1}}^2.
\eq
On the other hand,
\[
R = \nabla\cdot((1+h)u) = u\cdot\nabla h+(1+h)\nabla\cdot u,
\]
and thus
\[
{\mathsf I}_3^0 -\frac1\nu\intt (u\cdot\nabla h)^2\,dx = \frac1\nu \lt( \|R\|_{L^2}^2 - \|u\cdot\nabla h\|_{L^2}^2 \rt)  = \frac2\nu \intt  (u\cdot\nabla h)(1+h)\nabla\cdot u\,dx + \frac1\nu \|(1+h)\nabla\cdot u\|_{L^2}^2 .
\]
Using the density bootstrap, Sobolev embedding, and Young's inequality, we
obtain
\begin{align}\label{eq:I30-l}
\begin{aligned}
\lt| {\mathsf I}_3^0 -\frac1\nu\intt(u\cdot\nabla h)^2\,dx \rt| &\le \frac{1}{4\nu}\|\nabla h\|_{L^2}^2 + \frac{C}{\nu}\|u \nabla u\|_{L^2}^2 + \frac1\nu \|1+h\|_{L^\infty}\|\nabla u\|_{L^2}^2\cr
&\le \frac{1}{4\nu}\|\nabla h\|_{L^2}^2 + \frac{C(1+\calM_T)^2}{\nu} \|u\|_{H^{m+1}}^2 
\end{aligned}
\end{align}
Combining \eqref{eq:R0u} with \eqref{eq:I30-l} gives
\bq\label{eq:I20-I30}
{\mathsf I}_2^0 + {\mathsf I}_3^0 \le \frac{1}{2\nu}\|\nabla h\|_{L^2}^2 + \frac{C(1+\calM_T)^2}{\nu} \|u\|_{H^{m+1}}^2.
\eq
Collecting \eqref{eq:E0-id}, \eqref{eq:I10}, and \eqref{eq:I20-I30}, we obtain
\bq\label{eq:E0-diff}
\frac{d}{dt}\calE_0^h  + \frac{1}{2\nu}\|h\|_{H^1}^2 \le \frac{C(1+\calM_T)^2}{\nu}
\|u\|_{H^{m+1}}^2.
\eq

We compare $\calE_0^h$ with $\|h\|_{L^2}^2$. By Young's inequality and the product estimate,
\[
\lt| \frac2\nu \intt h\,R\,dx \rt| \le \frac12\|h\|_{L^2}^2 + \frac{C}{\nu^2}\|R\|_{L^2}^2 \le \frac12\|h\|_{L^2}^2 + \frac{C(1+\calM_T)^2\calM_T^2}{\nu^2}.
\]
This yields
\bq\label{eq:E0-equi}
\|h\|_{L^2}^2 \le 2\calE_0^h + \frac{C(1+\calM_T)^2\calM_T^2}{\nu^2}.
\eq
Integrating \eqref{eq:E0-diff} over $[0,t]$ and using
\[
\int_0^t\|u(s)\|_{H^{m+1}}^2\,ds \le \frac{\calM_T^2}{\nu},
\]
we find
\[
\calE_0^h(t) + \frac1\nu \int_0^t \|h(s)\|_{H^1}^2\,ds \le \calE_0^h(0) + \frac{C(1+\calM_T)^2\calM_T^2}{\nu^2}.
\]
Moreover,
\[
\calE_0^h(0) \le C\|h_0\|_{L^2}^2 + \frac{C(1+\|h_0\|_{H^1})^2\|u_0\|_{H^1}^2}{\nu^2}.
\]
Using \eqref{eq:E0-equi}, we therefore obtain
\[
\sup_{0\le t\le T}\|h(t)\|_{L^2}^2 + \frac1\nu \int_0^T \|h(t)\|_{H^1}^2\,dt \le C\|h_0\|_{L^2}^2 + \frac{C(1+\|h_0\|_{H^1})^2\|u_0\|_{H^1}^2}{\nu^2} + \frac{C(1+\calM_T)^2\calM_T^2}{\nu^2}.
\]
Under the large-damping bootstrap \eqref{eq:MTs}, the last term is bounded by $\delta_0^2$. Therefore, choosing $\delta_0>0$ fixed and sufficiently small, we may define
\[
C_0^h := C\lt( 1+\|h_0\|_{L^2}^2 + (1+\|h_0\|_{H^1})^2\|u_0\|_{H^1}^2 \rt),
\]
where $C$ depends only on $d,m,\rho_1,\rho_2$. Then
\bq\label{eq:den-base}
\sup_{0\le t\le T}\|h(t)\|_{L^2}^2 + \frac1\nu \int_0^T \|h(t)\|_{H^1}^2\,dt \le C_0^h.
\eq

\medskip

\noindent{\bf Step 2. Induction step.} Assume that, for some $1\le k\le m$, the estimate
\bq\label{eq:den-hyp}
\sup_{0\le t\le T}\|h(t)\|_{H^{k-1}}^2 + \frac1\nu \int_0^T \|h(t)\|_{H^k}^2\,dt \le C_{k-1}^h
\eq
holds. For $k=1$, this is exactly the base estimate \eqref{eq:den-base}. We prove the estimate at order $k$.

Let $|\alpha|=k$. From
\[
\pa^\alpha\bar h = \pa^\alpha h-\frac1\nu\pa^\alpha R,
\]
we estimate, as above,
\begin{align*}
\frac12\frac{d}{dt}\|\pa^\alpha\bar h\|_{L^2}^2 
&= \frac1\nu \intt \pa^\alpha h\, \pa^\alpha \lt[ \nabla\cdot((1+h)\nabla\phi) +\Delta h +\nabla\otimes\nabla:((1+h)u\otimes u) \rt]dx \cr
&\quad -\frac1{\nu^2} \intt \pa^\alpha R\, \pa^\alpha \lt[ \nabla\cdot((1+h)\nabla\phi) +\Delta h +\nabla\otimes\nabla:((1+h)u\otimes u) \rt]dx.
\end{align*}
Using again
\[
\nabla\cdot((1+h)\nabla\phi)+\Delta h +\nabla\otimes\nabla:((1+h)u\otimes u) = -\pa_t R-\nu R,
\]
we obtain
\bq\label{eq:alp-id}
\frac12\frac{d}{dt} \lt( \|\pa^\alpha\bar h\|_{L^2}^2 - \lt\| \frac1\nu\pa^\alpha R\rt\|_{L^2}^2 \rt) + \frac1\nu\|\pa^\alpha h\|_{H^1}^2 = {\mathsf I}_1^\alpha+{\mathsf I}_2^\alpha+{\mathsf I}_3^\alpha,
\eq
where
\[
{\mathsf I}_1^\alpha := \frac1\nu \intt \pa^\alpha h\, \pa^\alpha\nabla\cdot(h\nabla\phi)\,dx, \quad {\mathsf I}_2^\alpha := \frac1\nu \intt \pa^\alpha h\, \pa^\alpha\nabla\otimes\nabla:((1+h)u\otimes u)\,dx,
\]
and
\[
{\mathsf I}_3^\alpha := \frac1\nu\|\pa^\alpha R\|_{L^2}^2.
\]
Since
\[
\|\pa^\alpha\bar h\|_{L^2}^2 - \lt\| \frac1\nu\pa^\alpha R \rt\|_{L^2}^2 = \|\pa^\alpha h\|_{L^2}^2 - \frac2\nu \intt \pa^\alpha h\,\pa^\alpha R\,dx,
\]
this gives the modified density energy at order $k$:
\[
\calE_k^h := \|\nabla^k h\|_{L^2}^2 -\frac2\nu \sum_{|\alpha|=k} \intt \pa^\alpha h\, \pa^\alpha R\,dx.
\]

We now estimate the right-hand side.  For ${\mathsf I}_1^\alpha$, applying Lemma \ref{lem:po-drift} with $f=h$ and $\Phi=\phi$, we obtain
\bq\label{eq:I1-est}
|{\mathsf I}_1^\alpha| \le \frac{C}{\nu} \|\nabla^k h\|_{L^2}^2 \|\nabla^{k+1}h\|_{L^2}.
\eq
For ${\mathsf I}_2^\alpha$, integrating by parts once gives
\[
{\mathsf I}_2^\alpha = -\frac1\nu \sum_{i,j=1}^d \intt \pa_i\pa^\alpha h\, \pa^\alpha\pa_j\lt((1+h)u_i u_j\rt) dx.
\]
Splitting off the highest derivative falling on $h$,
\[
\pa^\alpha\pa_j\lt((1+h)u_i u_j\rt) = u_i u_j\,\pa_j\pa^\alpha h + \calR_{\alpha,ij}^{u},
\]
where
\[
\calR_{\alpha,ij}^{u} := \pa^\alpha\pa_j\lt((1+h)u_i u_j\rt) - u_i u_j\,\pa_j\pa^\alpha h.
\]
Then
\begin{align*}
{\mathsf I}_2^\alpha &= -\frac1\nu \sum_{i,j=1}^d \intt u_i u_j\, \pa_i\pa^\alpha h\, \pa_j\pa^\alpha h\,dx -\frac1\nu \sum_{i,j=1}^d \intt \pa_i\pa^\alpha h\, \calR_{\alpha,ij}^{u}\,dx \cr
&= -\frac1\nu \intt (u\cdot\nabla\pa^\alpha h)^2\,dx -\frac1\nu \sum_{i,j=1}^d \intt \pa_i\pa^\alpha h\, \calR_{\alpha,ij}^{u}\,dx.
\end{align*}
The first term is nonpositive and will be discarded.

For the remainder, using the product estimate and Sobolev embedding,
\[
\|u\|_{L^\infty} + \|u\|_{W^{1,\infty}} \le C\|u\|_{H^{m+1}}, \quad \|h\|_{L^\infty}\le C\|h\|_{H^m},
\]
we estimate
\begin{align*}
\|\calR_{\alpha,ij}^{u}\|_{L^2} &\le C\|1+h\|_{L^\infty}\|\nabla^{k+1}(u_i u_j)\|_{L^2} + C\|\nabla^k h\|_{L^2}\|\nabla(u_i u_j)\|_{L^\infty} \cr
&\le C \lt( 1+\|h\|_{L^\infty}+\|\nabla^k h\|_{L^2} \rt) \lt( \|u\|_{H^{k+1}}+\|u\|_{W^{1,\infty}} \rt) \|u\|_{W^{1,\infty}}\cr
&\le C\lt( 1+\|h\|_{H^m}\rt)\|u\|_{H^{m+1}}^2.
\end{align*}
Hence, we have
\bq\label{eq:I2-est}
{\mathsf I}_2^\alpha \le \frac{C}{\nu} \lt( 1+\|h\|_{H^m}  \rt) \|\nabla^{k+1}h\|_{L^2} \|u\|_{H^{m+1}}^2.
\eq
  
Finally, using the product estimate and the Sobolev inequality, we get
\begin{align*}
\|\pa^\alpha\nabla\cdot((1+h)u)\|_{L^2} &\le C\lt( \|\nabla^{k+1}h\|_{L^2}\|u\|_{L^\infty} + (1+\|h\|_{L^\infty})\|u\|_{H^{k+1}} \rt) \cr
&\le C\lt( 1+\|h\|_{H^m}+\|\nabla^{k+1} h\|_{L^2} \rt)\|u\|_{H^{m+1}},
\end{align*}
and thus
\bq\label{eq:I3-est}
{\mathsf I}_3^\alpha \le \frac{C}{\nu} \lt( 1+\|h\|_{H^m}+\|\nabla^{k+1} h\|_{L^2} \rt)^2\|u\|_{H^{m+1}}^2.
\eq

Collecting \eqref{eq:alp-id} with \eqref{eq:I1-est}, \eqref{eq:I2-est}, and \eqref{eq:I3-est}, and using the bootstrap bound, we obtain
\bq\label{est_h_high}
\frac{d}{dt}\calE_k^h + \frac{2}{\nu}\|\nabla^k h\|_{H^1}^2 \le \frac{C}{\nu}\|\nabla^k h\|_{L^2}^2\|\nabla^{k+1} h\|_{L^2} + \frac{C(1+\calM_T)^2}{\nu} \|u\|_{H^{m+1}}^2,
\eq
where we used 
\[
(1+\|h\|_{H^m})\|\nabla^{k+1}h\|_{L^2} + \lt(1+\|h\|_{H^m}+\|\nabla^{k+1}h\|_{L^2}\rt)^2 \le C(1+\calM_T)^2.
\]
By Young's inequality, we get

\[
\frac{d}{dt}\calE_k^h + \frac{1}{\nu}\|\nabla^k h\|_{H^1}^2 \le \frac{C}{\nu}\|\nabla^k h\|_{L^2}^4 + \frac{C(1+\calM_T)^2}{\nu} \|u\|_{H^{m+1}}^2.
\]

Similarly as before, we compare $\calE_k^h$ and $\|\nabla^k h\|_{L^2}^2$. Since
\bq\label{eq:Ek-equi}
\lt| \frac2\nu \sum_{|\alpha|=k} \intt \pa^\alpha h\, \pa^\alpha R\,dx \rt| \le \frac12\|\nabla^k h\|_{L^2}^2 + \frac{C}{\nu^2}\|R\|_{H^k}^2 \le \frac12\|\nabla^k h\|_{L^2}^2 + \frac{C(1+\calM_T)^2\calM_T^2}{\nu^2},
\eq
we get
\bq\label{eq:h-Ek}
\|\nabla^k h\|_{L^2}^2 \le 2\calE_k^h + \frac{C(1+\calM_T)^2\calM_T^2}{\nu^2}.
\eq
Using \eqref{eq:h-Ek}, we also find
\[
\|\nabla^k h\|_{L^2}^4 \le C\|\nabla^k h\|_{L^2}^2\calE_k^h + \frac{C(1+\calM_T)^2\calM_T^2}{\nu^2} \|\nabla^k h\|_{L^2}^2.
\]
Hence, we obtain
\bq\label{eq:Ek-diff-i}
\frac{d}{dt}\calE_k^h + \frac{1}{2\nu}\|\nabla^k h\|_{H^1}^2 \le \frac{C}{\nu} \|\nabla^k h\|_{L^2}^2\calE_k^h + \frac{C(1+\calM_T)^2\calM_T^2}{\nu^3} \|\nabla^k h\|_{L^2}^2 +\frac{C(1+\calM_T)^2}{\nu} \|u\|_{H^{m+1}}^2.
\eq
Applying Gronwall's lemma to \eqref{eq:Ek-diff-i}, we obtain
\begin{align}\label{eq:Ek-gron}
\begin{aligned}
&\calE_k^h(t) + \frac1\nu \int_0^t \|\nabla^k h(s)\|_{H^1}^2\,ds \cr
&\quad \le e^{C C_{k-1}^h} \lt( \calE_k^h(0) + \frac{C(1+\calM_T)^2\calM_T^2}{\nu^3} \int_0^t \|\nabla^k h(s)\|_{L^2}^2\,ds + \frac{C(1+\calM_T)^2}{\nu} \int_0^t \|u(s)\|_{H^{m+1}}^2\,ds \rt) \cr
&\quad \le e^{C C_{k-1}^h} \lt( \calE_k^h(0) + \frac{C(1+\calM_T)^2\calM_T^2 C_{k-1}^h}{\nu^2} + \frac{C(1+\calM_T)^2\calM_T^2}{\nu^2} \rt),
\end{aligned}
\end{align}
where we used the induction hypothesis,
\[
\frac1\nu \int_0^T \|\nabla^k h(s)\|_{L^2}^2\,ds \le C_{k-1}^h
\]
and the bootstrap bound,
\[
\int_0^t\|u(s)\|_{H^{m+1}}^2\,ds \le \frac{\calM_T^2}{\nu}.
\]
The initial modified energy satisfies
\bq\label{eq:initial-Ek}
\calE_k^h(0) \le C\|h_0\|_{H^k}^2 + \frac{C}{\nu^2}\|R(0)\|_{H^k}^2 \le
C\|h_0\|_{H^k}^2 + \frac{C(1+\|h_0\|_{H^{m+1}})^2 \|u_0\|_{H^{m+1}}^2}{\nu^2}.
\eq
Combining \eqref{eq:Ek-gron}, \eqref{eq:initial-Ek}, \eqref{eq:Ek-equi}, and \eqref{eq:den-hyp}, we
conclude that
\begin{align*}
&\sup_{0\le t\le T}\|h(t)\|_{H^k}^2 + \frac1\nu \int_0^T \|h(t)\|_{H^{k+1}}^2\,dt \cr
&\quad \le C_{k-1}^h + e^{C C_{k-1}^h} \lt(  C\|h_0\|_{H^k}^2 + \frac{C(1+\|h_0\|_{H^{m+1}})^2 \|u_0\|_{H^{m+1}}^2}{\nu^2} + \frac{C(1+\calM_T)^2\calM_T^2 (1+C_{k-1}^h)}{\nu^2}\rt).
\end{align*}
Under the large-damping bootstrap \eqref{eq:MTs}, choosing $\delta_0>0$ small enough, the last term is bounded by a constant depending only on $C_{k-1}^h$ and the initial data. Thus, we may define
\[
C_k^h := C_{k-1}^h + e^{C C_{k-1}^h}\lt( 1+ C_{k-1}^h  +\|h_0\|_{H^k}^2 + (1+\|h_0\|_{H^{m+1}})^2 \|u_0\|_{H^{m+1}}^2 \rt),
\]
Hence, we have
\[
\sup_{0\le t\le T}\|h(t)\|_{H^k}^2 + \frac1\nu\int_0^T\|h(t)\|_{H^{k+1}}^2\,dt \le C_k^h.
\]
The constant $C_k^h$ depends only on $d,m,\rho_1,\rho_2$, the initial data, and $C_{k-1}^h$, and is independent of $T$ and $\nu$.

Taking $k=m$, we obtain
\[
\sup_{0\le t\le T}\|h(t)\|_{H^m}^2 + \frac1\nu \int_0^T \|h(t)\|_{H^{m+1}}^2\,dt \le C_m^h.
\]
Setting $C_1:=C_m^h$ completes the proof.
\end{proof}

%
%
%
%
%

\subsection{Top-order hyperbolic estimate}
We now derive the top-order estimate. The main difficulty is that the pressure force
\[
\frac{\nabla h}{1+h}
\]
contains one full derivative of the density. The estimate is closed by using two weighted cancellations: the order-$m$ density estimate weighted by $(1+h)^{-1}$ cancels the Poisson force, while the order-$(m+1)$ density estimate weighted by $(1+h)^{-2}$ cancels the leading pressure term in the velocity equation.

We shall use the density estimate obtained in Proposition \ref{prop:mod-den}: there exists a constant $C_1>0$, independent of $T$ and $\nu$, such that
\bq\label{eq:C1-den-bdd}
\sup_{0\le t\le T}\|h(t)\|_{H^m}^2 + \frac1\nu\int_0^T\|h(t)\|_{H^{m+1}}^2\,dt \le C_1.
\eq

\begin{proposition} \label{prop:top-order}
Let $2\le d\le4$ and let $m\in\N$ satisfy $m>d/2$. Assume that \eqref{eq:den-boot}, \eqref{eq:MTs}, and \eqref{eq:C1-den-bdd} hold. Then, there exists a constant $C>0$ depending only on $d,m,\rho_1,\rho_2$ such that 
\[
\sup_{0\le t\le T} \lt( \|h(t)\|_{H^{m+1}}^2+\|u(t)\|_{H^{m+1}}^2 \rt) + \nu\int_0^T\|u(t)\|_{H^{m+1}}^2\,dt \le e^{CC_1}\lt( \|h_0\|_{H^{m+1}}^2+\|u_0\|_{H^{m+1}}^2 \rt).
\]
\end{proposition}

\begin{proof}
We first record the weighted density estimate in a form that will be used twice. Let $|\gamma|=r$ and $\ell\in\N$. Multiplying the differentiated continuity equation by $(1+h)^{-\ell}\pa^\gamma h$, we have
\[
\frac12\frac{d}{dt} \intt(1+h)^{-\ell}|\pa^\gamma h|^2\,dx = -\frac{\ell}{2} \intt (1+h)^{-\ell-1}\pa_t h |\pa^\gamma h|^2\,dx + \intt (1+h)^{-\ell}\pa^\gamma h\,\pa^\gamma \pa_t h \,dx .
\]
Using
\[
\pa_t h =-\nabla h\cdot u-(1+h)\nabla\cdot u,
\]
we obtain
\[
-\frac{\ell}{2} \intt (1+h)^{-\ell-1}\pa_t h |\pa^\gamma h|^2\,dx = \frac{\ell}{2} \intt (1+h)^{-\ell-1}(\nabla h\cdot u)|\pa^\gamma h|^2\,dx + \frac{\ell}{2} \intt (1+h)^{-\ell}(\nabla\cdot u)|\pa^\gamma h|^2\,dx .
\]
On the other hand,
\[
\intt (1+h)^{-\ell}\pa^\gamma h\,\pa^\gamma \pa_t h \,dx = -\intt (1+h)^{-\ell}\pa^\gamma h\,\pa^\gamma\nabla\cdot u\,dx - \intt (1+h)^{-\ell}\pa^\gamma h\,\pa^\gamma\nabla\cdot(hu)\,dx.
\]
We decompose
\[
\pa^\gamma\nabla\cdot(hu) = h\,\pa^\gamma\nabla\cdot u + u\cdot\nabla\pa^\gamma h + \calC_\gamma,
\]
where
\[
\calC_\gamma := \pa^\gamma\nabla\cdot(hu) - h\,\pa^\gamma\nabla\cdot u - u\cdot\nabla\pa^\gamma h .
\]
Then
\begin{align*}
&-\intt (1+h)^{-\ell}\pa^\gamma h\,\pa^\gamma\nabla\cdot u\,dx - \intt (1+h)^{-\ell}\pa^\gamma h\, h\,\pa^\gamma\nabla\cdot u\,dx \cr
&\quad = -\intt (1+h)^{-\ell+1}\pa^\gamma h\, \pa^\gamma\nabla\cdot u\,dx \cr
&\quad = \intt (1+h)^{-\ell+1}\nabla\pa^\gamma h\cdot \pa^\gamma u\,dx -(\ell-1) \intt (1+h)^{-\ell} \pa^\gamma h\,\nabla h\cdot \pa^\gamma u\,dx .
\end{align*}
We also find
\begin{align*}
-\intt (1+h)^{-\ell}\pa^\gamma h\,u\cdot\nabla\pa^\gamma h\,dx &= -\frac12 \intt (1+h)^{-\ell}u\cdot\nabla |\pa^\gamma h|^2\,dx \cr
&= \frac12 \intt \nabla\cdot\lt((1+h)^{-\ell}u\rt) |\pa^\gamma h|^2\,dx \cr
&= -\frac{\ell}{2} \intt(1+h)^{-\ell-1}(\nabla h\cdot u)|\pa^\gamma h|^2\,dx +  \frac12\intt(1+h)^{-\ell}(\nabla\cdot u)|\pa^\gamma h|^2\,dx .
\end{align*}
Combining the above identities gives
\bq\label{eq:w-den-i}
\frac12\frac{d}{dt}
\|(1+h)^{-\frac\ell2}\pa^\gamma h\|_{L^2}^2 = \intt (1+h)^{-\ell+1}\nabla\pa^\gamma h\cdot \pa^\gamma u\,dx + \calR_\gamma^\ell ,
\eq
where
\begin{align*}
\calR_\gamma^\ell &:= \frac{\ell+1}2\intt  (1+h)^{-\ell}  (\nabla \cdot u)|\pa^\gamma h|^2\,dx -(\ell - 1) \intt (1+h)^{-\ell} (\pa^\gamma h)(\nabla h) \cdot \pa^\gamma u\,dx\cr
&\quad -\intt (1+h)^{-\ell} \pa^\gamma h \cdot \lt[ \pa^\gamma \nabla \cdot (hu) - \pa^\gamma \nabla h \cdot u - h \pa^\gamma \nabla \cdot u\rt] dx.
\end{align*}
By the standard product commutator estimate, the last term in $\calR_\gamma^\ell$ is controlled by
\[
\|\pa^\gamma h\|_{L^2} \lt( \|\nabla u\|_{L^\infty}\|\pa^\gamma h\|_{L^2} + \|\nabla h\|_{L^\infty}\|\pa^\gamma u\|_{L^2} \rt).
\]
Thus, we have
\[
|\calR_\gamma^\ell| \le C\|(1+h)^{-\ell}\|_{L^\infty}\|\pa^\gamma h\|_{L^2} \lt( \|\nabla u\|_{L^\infty}\|\pa^\gamma h\|_{L^2} + \|\nabla h\|_{L^\infty}\|\pa^\gamma u\|_{L^2} \rt).
\]

We now apply \eqref{eq:w-den-i} in two different ways. First, taking $\ell=1$ and $|\alpha|=m$, we get
\bq\label{eq:den-po-c}
\frac12\frac{d}{dt}
\|(1+h)^{-1/2}\pa^\alpha h\|_{L^2}^2 = \intt \nabla\pa^\alpha h\cdot \pa^\alpha u\,dx + \calR_\alpha^1 .
\eq
Second, taking $\ell=2$ and $\gamma=\alpha+e_j$, with $|\alpha|=m$ and
$j=1,\dots,d$, we find
\bq\label{eq:den-pr-c}
\frac12\frac{d}{dt} \|(1+h)^{-1}\pa_j\pa^\alpha h\|_{L^2}^2 = \intt (1+h)^{-1} \nabla\pa_j\pa^\alpha h\cdot \pa_j\pa^\alpha u\,dx + \calR_{\alpha,j}^2 .
\eq

We next estimate the velocity equation. Applying $\nabla\pa^\alpha$, $|\alpha|=m$, to
\[
\pa_t u+u\cdot\nabla u+(1+h)^{-1}\nabla h+\nabla\phi=-\nu u,
\]
and testing by $\nabla\pa^\alpha u$, we obtain
\begin{align}\label{eq:vel-top}
\begin{aligned}
&\frac12\frac{d}{dt}\intt |\nabla \pa^\alpha u|^2\,dx + \nu \intt |\nabla \pa^\alpha u|^2\,dx  \cr
&\quad = - \intt \nabla \pa^\alpha u : (u \cdot \pa^\alpha \nabla^2 u)\,dx - \intt \nabla \pa^\alpha u : \lt[ \nabla \pa^\alpha (u \cdot \nabla u) - u \cdot \pa^\alpha \nabla^2 u \rt] dx \cr
&\quad \quad - \intt (1+h)^{-1} \nabla \pa^\alpha u : \pa^\alpha \nabla^2 h\,dx  - \intt \nabla \pa^\alpha u : \lt[ \nabla \pa^\alpha ((1+h)^{-1}\nabla h) - (1+h)^{-1} \pa^\alpha \nabla^2 h \rt] dx \cr
&\quad \quad - \intt \nabla \pa^\alpha u : \pa^\alpha \nabla^2 \phi\,dx \cr
&\quad \le C\|\nabla u\|_{L^\infty}\|\nabla \pa^\alpha u\|_{L^2}^2   - \intt (1+h)^{-1} \nabla \pa^\alpha u : \pa^\alpha \nabla^2 h\,dx \cr
&\quad \quad + C\|\nabla \pa^\alpha u\|_{L^2} \lt( \|\nabla (1+h)^{-1}\|_{L^\infty} \|\nabla \pa^\alpha h\|_{L^2} + \|\nabla h\|_{L^\infty}\|\nabla \pa^\alpha (1+h)^{-1}\|_{L^2}\rt) \cr
&\quad \quad - \intt \nabla \pa^\alpha u : \pa^\alpha \nabla^2 \phi\,dx\cr
&\quad \le C\|\nabla u\|_{L^\infty}\|\nabla^{m+1}u\|_{L^2}^2 + C\|\nabla^{m+1}u\|_{L^2} \|\nabla h\|_{L^\infty}\|\nabla^{m+1}h\|_{L^2}\cr
&\quad \quad - \intt (1+h)^{-1} \nabla \pa^\alpha u : \pa^\alpha \nabla^2 h\,dx - \intt \nabla \pa^\alpha u : \pa^\alpha \nabla^2 \phi\,dx,
\end{aligned}
\end{align}
where we used Lemma \ref{lem:pow-est}, Sobolev's inequality, Poincar\'e's inequality, and the lower-order bounds. The leading pressure term can be written as
\[
-\sum_{j=1}^d \intt (1+h)^{-1} \pa_j\pa^\alpha u\cdot \nabla\pa_j\pa^\alpha h\,dx ,
\]
and thus it exactly cancels the leading term in \eqref{eq:den-pr-c}, after summing over $j=1,\dots,d$. The Poisson term cancels with \eqref{eq:den-po-c}. Indeed, using $-\Delta\phi=h$ and the periodic boundary condition,
\[
-\intt \nabla\pa^\alpha u: \pa^\alpha\nabla^2\phi\,dx = \intt \pa^\alpha u\cdot \pa^\alpha\nabla\Delta\phi\,dx = -\intt \pa^\alpha u\cdot\nabla\pa^\alpha h\,dx .
\]
This exactly cancels the leading term in \eqref{eq:den-po-c}.

Summing \eqref{eq:den-po-c} and \eqref{eq:vel-top} over $|\alpha|=m$, and summing \eqref{eq:den-pr-c} over $|\alpha|=m$ and $j=1,\dots,d$, the two leading terms cancel as above. We obtain
\bq\label{est_u_high}\begin{aligned}
\frac12\frac{d}{dt}\calE_{m+1}  + \nu\|\nabla^{m+1}u\|_{L^2}^2 &\le C\|\nabla u\|_{L^\infty}\|\nabla^{m+1} u\|_{L^2}^2 \\
&\quad + C\lt(\|\nabla u\|_{L^\infty} + \|\nabla h\|_{L^\infty}\rt)\|\nabla^{m+1} u\|_{L^2}\|\nabla^{m+1} h\|_{L^2}\\
&\le C\|\nabla^{m+1}u\|_{L^2}^3 + C\|\nabla^{m+1}u\|_{L^2} \|\nabla^{m+1}h\|_{L^2}^2,
\end{aligned}\eq
where
\[
\calE_{m+1} := \sum_{|\alpha|=m} \|(1+h)^{-1/2}\pa^\alpha h\|_{L^2}^2 + \sum_{|\alpha|=m}\sum_{j=1}^d \|(1+h)^{-1}\pa_j\pa^\alpha h\|_{L^2}^2 + \sum_{|\alpha|=m} \|\nabla\pa^\alpha u\|_{L^2}^2 . 
\]
By the pointwise density bounds \eqref{eq:den-boot}, $\calE_{m+1}$ is equivalent to the corresponding top-order quantity, namely
\[
\|\nabla^m h\|_{L^2}^2 + \|\nabla^{m+1}h\|_{L^2}^2 + \|\nabla^{m+1}u\|_{L^2}^2 .
\]

We now absorb the nonlinear velocity contribution. By Young's inequality, we get
\begin{align*}
C\|\nabla^{m+1}u\|_{L^2}^3 + C\|\nabla^{m+1}u\|_{L^2} \|\nabla^{m+1}h\|_{L^2}^2  &\le \frac{\nu}{2}\|\nabla^{m+1}u\|_{L^2}^2 + \frac{C}{\nu} \lt( \|\nabla^{m+1}u\|_{L^2}^4 + \|\nabla^{m+1}h\|_{L^2}^4 \rt) \cr
&\le \lt( \frac{\nu}{2} + \frac{C\calM_T^2}{\nu} \rt) \|\nabla^{m+1}u\|_{L^2}^2 + \frac{C}{\nu} \|\nabla^{m+1}h\|_{L^2}^2 \calE_{m+1}.
\end{align*}
Using the smallness condition \eqref{eq:MTs} in the large-damping regime, and choosing $\delta_0>0$ sufficiently small, the term involving $\|\nabla^{m+1}u\|_{L^2}^2$ can be absorbed into the damping term. Hence, we have
\[
\frac12\frac{d}{dt}\calE_{m+1} + c_0\nu\|\nabla^{m+1}u\|_{L^2}^2 \le \frac{C}{\nu} \|\nabla^{m+1}h\|_{L^2}^2 \calE_{m+1},
\]
for some $c_0>0$. Applying Gronwall's lemma and using the density estimate \eqref{eq:C1-den-bdd}, we obtain
\begin{align*}
\calE_{m+1}(t) + \nu\int_0^t \|\nabla^{m+1}u(s)\|_{L^2}^2\,ds &\le \calE_{m+1}(0) \exp\lt(\frac{C}{\nu} \int_0^t \|\nabla^{m+1}h(s)\|_{L^2}^2\,ds \rt) \cr
&\le e^{CC_1} \lt( \|h_0\|_{H^{m+1}}^2 + \|u_0\|_{H^{m+1}}^2 \rt).
\end{align*}
 
Finally, we pass from the homogeneous top-order estimate to the full $H^{m+1}$ estimate. The density part is obtained by combining the lower-order bound \eqref{eq:C1-den-bdd} with the control of $\|\nabla^{m+1}h\|_{L^2}$ contained in $\calE_{m+1}$. For the velocity, the top-order energy controls $\|\nabla^{m+1}u\|_{L^2}$, while the zeroth-order norm is controlled by Lemma \ref{lem:be}, since the density bootstrap bound gives positive lower and upper bounds for $\rho$. Thus, using the interpolation
\[
\|u\|_{H^{m+1}}^2 \le C\lt( \|u\|_{L^2}^2+\|\nabla^{m+1}u\|_{L^2}^2 \rt),
\]
we obtain
\[
\sup_{0\le t\le T} \lt( \|h(t)\|_{H^{m+1}}^2+\|u(t)\|_{H^{m+1}}^2 \rt) + \nu\int_0^T\|u(t)\|_{H^{m+1}}^2\,dt \le e^{CC_1}\lt( \|h_0\|_{H^{m+1}}^2+\|u_0\|_{H^{m+1}}^2 \rt).
\]
This completes the proof. 
\end{proof}

%
%
%
%
%

\section{Auxiliary density and comparison estimate}\label{sec:aux-comp}

In this section, we introduce an auxiliary drift-diffusion density and compare it with the Euler--Poisson density. The role of the auxiliary system is twofold. First, it preserves the pointwise lower and upper bounds of the initial density by a maximum principle. Second, it provides a parabolic reference dynamics to which the Euler--Poisson density can be compared in the large-damping regime.

Let $\tilde\rho=\tilde\rho(t,x)$ denote the auxiliary density, and set
\[
\tilde h:=\tilde\rho-1.
\]
The comparison variable is
\[
w:=h-\tilde h.
\]
As in the modified density estimate of Section \ref{sec:apri}, the comparison is not performed directly on $w$. Instead, we use the modified comparison variable
\bq\label{eq:Z-def}
Z := w-\frac1\nu\nabla\cdot(\rho u) = h-\tilde h-\frac1\nu\nabla\cdot((1+h)u).
\eq
This modification removes the time derivative of $\rho u$ in the parabolic reformulation of the density equation.

Throughout this section, we use the bootstrap quantity $\calM_T$ defined in \eqref{eq:MT}. We also use the estimates obtained in Section \ref{sec:apri}:
\bq\label{eq:den-bdd1}
\sup_{0\le t\le T}\|h(t)\|_{H^m}^2 + \frac1\nu\int_0^T\|h(t)\|_{H^{m+1}}^2\,dt \le C_1,
\eq
and
\bq\label{eq:top-bdd}
\sup_{0\le t\le T} \lt( \|h(t)\|_{H^{m+1}}^2+\|u(t)\|_{H^{m+1}}^2 \rt) + \nu\int_0^T\|u(t)\|_{H^{m+1}}^2\,dt \le C_2.
\eq
The constants $C_1>0$ and $C_2>0$ are independent of $T$ and $\nu$ in the large-damping regime.

%
%
%
%
%

\subsection{Auxiliary drift-diffusion system}

We first study the auxiliary drift-diffusion system. Let $\tilde\rho=\tilde\rho(t,x)$ and $\tilde\phi=\tilde\phi(t,x)$ solve
\bq\label{eq:aux-sys}
\pa_t\tilde\rho -\frac1\nu\nabla\cdot(\tilde\rho\nabla\tilde\phi) = \frac1\nu\Delta\tilde\rho, \quad -\Delta\tilde\phi=\tilde\rho-1,
\eq
with initial data
\[
\tilde\rho(0,x)=\rho_0(x).
\]
Equivalently, in terms of $\tilde h=\tilde\rho-1$, we have
\bq\label{eq:aux-h-sys}
\pa_t\tilde h -\frac1\nu\nabla\cdot(\tilde\rho\nabla\tilde\phi) = \frac1\nu\Delta\tilde h, \quad -\Delta\tilde\phi=\tilde h, \quad \tilde\rho=1+\tilde h.
\eq

The key point is that \eqref{eq:aux-sys} preserves the pointwise lower and upper bounds of the initial density.

\begin{lemma}\label{lem:aux-maximum}
Assume that
\[
0<\rho_1\le \rho_0(x)\le \rho_2, \quad x\in\T^d,
\]
and
\[
\intt \rho_0\,dx=1.
\]
Then every smooth solution $\tilde\rho$ to \eqref{eq:aux-sys} satisfies
\bq\label{eq:aux-maxi-bdd}
0<\rho_1\le \tilde\rho(t,x)\le \rho_2, \quad t\ge0,\quad x\in\T^d,
\eq
as long as the smooth solution exists.
\end{lemma}
\begin{proof}
Mass conservation follows directly from \eqref{eq:aux-sys}, and thus
\[
\intt \tilde\rho(t,x)\,dx=1.
\]
Let $x_t$ be a point where $\tilde\rho(t,\cdot)$ attains its maximum. At $(t,x_t)$, we have $\nabla\tilde\rho=0$, $\Delta\tilde\rho\le0$, and, by the mass constraint, $\tilde\rho\ge1$. Therefore, using $-\Delta\tilde\phi=\tilde\rho-1$,
\[
\pa_t\tilde\rho = \frac1\nu\nabla\cdot(\tilde\rho\nabla\tilde\phi) +\frac1\nu\Delta\tilde\rho = -\frac1\nu\tilde\rho(\tilde\rho-1) +\frac1\nu\Delta\tilde\rho \le0.
\]
Thus, the spatial maximum is nonincreasing.

The same argument applied at a spatial minimum, where $\Delta\tilde\rho\ge0$ and $\tilde\rho\le1$, shows that the minimum is nondecreasing. Therefore, the bounds \eqref{eq:aux-maxi-bdd} follow from the initial bounds.
\end{proof}

We next derive Sobolev estimates for $\tilde h$. The following estimate is the one needed later in the comparison argument.
\begin{proposition}\label{prop:aux-sob}
Let $2\le d\le4$ and let $m\in\N$ satisfy $m>d/2$. Assume the hypotheses of Lemma \ref{lem:aux-maximum}. Then, for each integer $0\le k\le m+1$, the solution $\tilde h=\tilde\rho-1$ satisfies
\[
\|\tilde h(t)\|_{H^k}^2 + \frac1\nu\int_0^t \|\tilde h(s)\|_{H^{k+1}}^2\,ds \le \tilde C_k, \quad t\ge0,
\]
where the constants $\tilde C_k$ are defined recursively by
\[
\tilde C_0:=\|h_0\|_{L^2}^2,
\]
and, for $1\le k\le m+1$,
\bq\label{eq:Ctildek-def}
\tilde C_k := \tilde C_{k-1} +  e^{C\tilde C_{k-1}} \|\nabla^k h_0\|_{L^2}^2.
\eq
Here $C>0$ depends only on $d,k,\rho_1,\rho_2$, but is independent of $t$ and $\nu$.
\end{proposition}

\begin{proof}
We first prove the $L^2$ estimate. Multiplying the system \eqref{eq:aux-h-sys} by $\tilde h$ and integrating over $\T^d$, we obtain
\[
\frac12\frac{d}{dt}\|\tilde h\|_{L^2}^2 + \frac1\nu\|\nabla\tilde h\|_{L^2}^2 = \frac1\nu \intt  \tilde h\,\nabla\cdot(\tilde\rho\nabla\tilde\phi)\,dx.
\]
Since $\tilde\rho=1+\tilde h$ and $-\Delta\tilde\phi=\tilde h$, we compute
\[
\intt  \tilde h\,\nabla\cdot(\tilde\rho\nabla\tilde\phi)\,dx =  \frac12\intt \tilde\rho^2\Delta\tilde\phi\,dx = -\frac12\intt  \tilde\rho^2(\tilde\rho-1)\,dx.
\]
Using $\intt \tilde h\,dx=0$, we have
\[
-\frac12\intt  \tilde\rho^2(\tilde\rho-1)\,dx = -\frac12\intt  (1+\tilde h)^2\tilde h\,dx.
\]
Moreover,
\[
(1+\tilde h)^2\tilde h = (1+\tilde h)\tilde h^2 + (1+\tilde h)\tilde h,
\]
and
\[
\intt (1+\tilde h)\tilde h\,dx = \intt (\tilde h+\tilde h^2)\,dx = \|\tilde h\|_{L^2}^2.
\]
Thus, we obtain
\[
\intt  \tilde h\,\nabla\cdot(\tilde\rho\nabla\tilde\phi)\,dx = -\frac12\intt  (1+\tilde h)\tilde h^2\,dx -\frac12\|\tilde h\|_{L^2}^2.
\]
By the maximum principle,
\[
1+\tilde h=\tilde\rho\ge \rho_1>0,
\]
and hence
\[
\frac{d}{dt}\|\tilde h\|_{L^2}^2 + \frac2\nu\|\nabla\tilde h\|_{L^2}^2 + \frac1\nu\|\tilde h\|_{L^2}^2 \le 0.
\]
Consequently, we have
\[
\|\tilde h(t)\|_{L^2}^2 + \frac1\nu\int_0^t\|\tilde h(s)\|_{H^1}^2\,ds \le \|h_0\|_{L^2}^2.
\]
This proves the estimate for $k=0$.

We now prove the higher-order estimate. Let $1\le k\le m+1$, and assume that
\bq\label{eq:h-ind}
\|\tilde h(t)\|_{H^{k-1}}^2 + \frac1\nu\int_0^t \|\tilde h(s)\|_{H^k}^2\,ds \le \tilde C_{k-1}.
\eq
We prove the estimate at order $k$.
 
Let $|\alpha|=k$. Applying $ \pa^\alpha$ to \eqref{eq:aux-h-sys}, multiplying by $ \pa^\alpha\tilde h$, and integrating over $\T^d$, we obtain
\bq\label{eq:aux-Hk0}
\frac12\frac{d}{dt}\| \pa^\alpha\tilde h\|_{L^2}^2 + \frac1\nu\|\nabla \pa^\alpha\tilde h\|_{L^2}^2 = \frac1\nu \intt  \pa^\alpha\tilde h\, \pa^\alpha\nabla\cdot(\tilde\rho\nabla\tilde\phi)\,dx.
\eq
Since
\[
\nabla\cdot(\tilde\rho\nabla\tilde\phi) = -\tilde h+\nabla\cdot(\tilde h\nabla\tilde\phi),
\]
we have
\bq\label{eq:aux-sp}
\frac1\nu \intt  \pa^\alpha\tilde h\, \pa^\alpha\nabla\cdot(\tilde\rho\nabla\tilde\phi)\,dx = -\frac1\nu\| \pa^\alpha\tilde h\|_{L^2}^2 + \frac1\nu \intt   \pa^\alpha\tilde h\,  \pa^\alpha\nabla\cdot(\tilde h\nabla\tilde\phi)\,dx.
\eq
The nonlinear drift term is estimated by Lemma \ref{lem:po-drift}, applied with $f=\tilde h$ and $\Phi=\tilde\phi$:
\[
\lt| \frac1\nu \intt   \pa^\alpha\tilde h\, \pa^\alpha\nabla\cdot(\tilde h\nabla\tilde\phi)\,dx\rt| \le \frac{C}{\nu} \|\nabla^k\tilde h\|_{L^2}^2 \|\nabla^{k+1}\tilde h\|_{L^2}.
\]
Using Young's inequality, we get
\bq\label{eq:aux-non}
\lt| \frac1\nu \intt   \pa^\alpha\tilde h\,  \pa^\alpha\nabla\cdot(\tilde h\nabla\tilde\phi)\,dx \rt| \le \frac{1}{2\nu}\|\nabla^{k+1}\tilde h\|_{L^2}^2 + \frac{C}{\nu}\|\nabla^k\tilde h\|_{L^2}^4.
\eq
Combining \eqref{eq:aux-Hk0}, \eqref{eq:aux-sp}, and \eqref{eq:aux-non}, and summing over $|\alpha|=k$, we obtain
\[
\frac{d}{dt}\|\nabla^k\tilde h\|_{L^2}^2 + \frac1\nu\|\nabla^{k+1}\tilde h\|_{L^2}^2 + \frac{c}{\nu}\|\nabla^k\tilde h\|_{L^2}^2 \le \frac{C}{\nu}\|\nabla^k\tilde h\|_{L^2}^4.
\]

We now close the induction. Applying Gr\"onwall's lemma, we have
\[
\|\nabla^k\tilde h(t)\|_{L^2}^2 + \frac1\nu\int_0^t \|\nabla^{k+1}\tilde h(s)\|_{L^2}^2\,ds \le \|\nabla^k h_0\|_{L^2}^2\exp\lt(\frac{C}{\nu}\int_0^t  \|\nabla^k\tilde h(s)\|_{L^2}^2\,ds \rt) \le \|\nabla^k h_0\|_{L^2}^2 e^{C \tilde C_{k-1}}
\]
due to the induction hypothesis \eqref{eq:h-ind}. Adding the lower-order estimate controlled by $\tilde C_{k-1}$ yields
\[
\|\tilde h(t)\|_{H^k}^2 + \frac1\nu\int_0^t \|\tilde h(s)\|_{H^{k+1}}^2\,ds \le \tilde C_k,
\]
where $\tilde C_k$ is defined by \eqref{eq:Ctildek-def}. This closes the induction and proves the proposition.
\end{proof}

%
%
%
%
%

\subsection{Comparison with the auxiliary density}

We now compare the Euler--Poisson density with the auxiliary density. The estimate is based on the modified comparison variable $Z$ defined in \eqref{eq:Z-def}. From Proposition \ref{prop:aux-sob}, the auxiliary solution satisfies
\bq\label{eq:tilde-bdd}
\|\tilde h(t)\|_{H^{m+1}}^2 + \frac1\nu\int_0^t \|\tilde h(s)\|_{H^{m+2}}^2\,ds \le \tilde C_{m+1}, \quad t\ge0.
\eq
The following estimate shows that $h$ remains close to $\tilde h$ in $H^m$ when the damping is large.

\begin{proposition}\label{prop:comp-den}
Let $2\le d\le4$ and let $m\in\N$ satisfy $m>d/2$. Assume that \eqref{eq:den-boot}, \eqref{eq:den-bdd1}, \eqref{eq:top-bdd}, and \eqref{eq:tilde-bdd} hold.
Then, we have
\[
\sup_{0\le t\le T}\|w(t)\|_{H^m}^2 + \frac1\nu\int_0^T\|w(t)\|_{H^{m+1}}^2\,dt \le  \frac{C(1+\calM_T)^2(1+\calM_T^2)}{\nu^2}.
\]
Here $C>0$ depends only on the initial data, $d,m, \rho_1, \rho_2$, but is independent of $T$, $\nu$, and $\calM_T$.
\end{proposition}

\begin{proof}
We first derive the equation for the modified comparison variable. From the
momentum equation,
\[
\nu\rho u = -\pa_t(\rho u) -\nabla\cdot(\rho u\otimes u) -\nabla h -\rho\nabla\phi.
\]
Taking divergence and using
\[
\pa_t h=-\nabla\cdot(\rho u),
\]
we obtain
\bq\label{eq:hp-comp}
\pa_t h - \frac1\nu\nabla\cdot(\rho\nabla\phi) = \frac1\nu\Delta h + \frac1\nu\nabla\cdot \lt[ \pa_t(\rho u) + \nabla\cdot(\rho u\otimes u) \rt].
\eq
On the other hand, the auxiliary density satisfies
\bq\label{eq:thp-comp}
\pa_t\tilde h - \frac1\nu\nabla\cdot(\tilde\rho\nabla\tilde\phi) = \frac1\nu\Delta\tilde h.
\eq
Subtracting \eqref{eq:thp-comp} from \eqref{eq:hp-comp}, and using \eqref{eq:Z-def}, we get
\bq\label{eq:Z-eq}
\pa_t Z = \frac1\nu \nabla\cdot \lt( \rho\nabla\phi-\tilde\rho\nabla\tilde\phi \rt) + \frac1\nu\Delta w + \frac1\nu\nabla\otimes\nabla:(\rho u\otimes u).
\eq

Let
\[
\psi:=\phi-\tilde\phi, \quad -\Delta\psi=w.
\]
We write \eqref{eq:Z-eq} in the compact form
\[
\pa_t Z = \frac1\nu\calF,
\]
where
\bq\label{eq:F-def}
\calF := \nabla\cdot(\rho\nabla\phi-\tilde\rho\nabla\tilde\phi) + \Delta w + \nabla\otimes\nabla:(\rho u\otimes u).
\eq
Since
\[
\rho\nabla\phi-\tilde\rho\nabla\tilde\phi = \rho\nabla\psi+w\nabla\tilde\phi = \nabla\psi+h\nabla\psi+w\nabla\tilde\phi,
\]
we have
\bq\label{eq:F-ex}
\calF = -w+\Delta w + \nabla\cdot(h\nabla\psi) + \nabla\cdot(w\nabla\tilde\phi) + \nabla\otimes\nabla:(\rho u\otimes u).
\eq

Let $|\alpha|=m$ and set
\[
R:=\nabla\cdot(\rho u).
\]
Since
\[
 \pa^\alpha Z =  \pa^\alpha w-\frac1\nu \pa^\alpha R,
\]
we obtain
\bq\label{eq:Z-en}
\frac12\frac{d}{dt}\| \pa^\alpha Z\|_{L^2}^2 = \frac1\nu \intt   \pa^\alpha w\, \pa^\alpha\calF\,dx - \frac1{\nu^2} \intt   \pa^\alpha R\, \pa^\alpha\calF\,dx .
\eq
Using \eqref{eq:F-ex}, the first term gives the principal dissipation
\[
- \frac1\nu\| \pa^\alpha w\|_{L^2}^2 - \frac1\nu\|\nabla \pa^\alpha w\|_{L^2}^2,
\]
together with the nonlinear terms involving $h\nabla\psi$, $w\nabla\tilde\phi$, and $\rho u\otimes u$.

For the second term in \eqref{eq:Z-en}, we use another representation of $\calF$. Taking divergence in the momentum equation,
we find
\[
\nabla\cdot(\rho\nabla\phi) + \Delta h + \nabla\otimes\nabla:(\rho u\otimes u) = -\pa_t R-\nu R .
\]
Thus, by \eqref{eq:F-def},
\bq\label{eq:F2}
\calF = -\pa_t R-\nu R - \nabla\cdot(\tilde\rho\nabla\tilde\phi) - \Delta\tilde h .
\eq
Substituting \eqref{eq:F2} into the second term in \eqref{eq:Z-en}, we obtain
\bq\label{eq:Rt}
-\frac1{\nu^2} \intt \pa^\alpha R\, \pa^\alpha\calF\,dx = \frac12\frac{d}{dt} \lt\| \frac1\nu \pa^\alpha R \rt\|_{L^2}^2 + \frac1\nu \| \pa^\alpha R\|_{L^2}^2 + \frac1{\nu^2} \intt  \pa^\alpha R\,  \pa^\alpha\nabla\cdot \lt( \tilde\rho\nabla\tilde\phi+\nabla\tilde h \rt)dx .
\eq

Combining \eqref{eq:Z-en} and \eqref{eq:Rt}, we get
\[
\frac12\frac{d}{dt} \lt( \| \pa^\alpha Z\|_{L^2}^2 - \frac1{\nu^2}\lt\|  \pa^\alpha R \rt\|_{L^2}^2 \rt) + \frac1\nu\| \pa^\alpha w\|_{L^2}^2 + \frac1\nu\|\nabla \pa^\alpha w\|_{L^2}^2 =
{\mathsf J}_1^\alpha+{\mathsf J}_2^\alpha+{\mathsf J}_3^\alpha+{\mathsf J}_4^\alpha,
\]
where
\begin{align*}
{\mathsf J}_1^\alpha &:= \frac1\nu \intt   \pa^\alpha w\, \pa^\alpha\nabla\cdot(h\nabla\psi)\,dx + \frac1\nu \intt   \pa^\alpha w\, \pa^\alpha\nabla\cdot(w\nabla\tilde\phi)\,dx,\cr
{\mathsf J}_2^\alpha &:= \frac1\nu \intt   \pa^\alpha w\,  \pa^\alpha\nabla\otimes\nabla:(\rho u\otimes u)\,dx, \quad {\mathsf J}_3^\alpha := \frac1\nu\| \pa^\alpha R\|_{L^2}^2,\cr
{\mathsf J}_4^\alpha &:= \frac1{\nu^2} \intt   \pa^\alpha R\,  \pa^\alpha\nabla\cdot \lt( \tilde\rho\nabla\tilde\phi+\nabla\tilde h \rt)dx.
\end{align*}

Since
\[
\| \pa^\alpha Z\|_{L^2}^2 - \frac1{\nu^2}\lt\|  \pa^\alpha R \rt\|_{L^2}^2 = \| \pa^\alpha w\|_{L^2}^2 - \frac2\nu \intt   \pa^\alpha w\, \pa^\alpha R\,dx,
\]
we define the modified comparison energy
\[
\calE_m^w(t) := \|\nabla^m w(t)\|_{L^2}^2 - \frac2\nu \sum_{|\alpha|=m} \intt  \pa^\alpha w\, \pa^\alpha R\,dx .
\]

We now estimate the terms on the right-hand side. For ${\mathsf J}_1^\alpha$, integrating by parts and using H\"older's inequality give
\[
|{\mathsf J}_1^\alpha| \le \frac1\nu \|\nabla\pa^\alpha w\|_{L^2} \lt( \|\pa^\alpha(h\nabla\psi)\|_{L^2} + \|\pa^\alpha(w\nabla\tilde\phi)\|_{L^2} \rt).
\]
Using the product estimate and elliptic regularity,
\[
\|\nabla\psi\|_{H^m} \le C\|w\|_{H^{m-1}} \le C\|w\|_{H^m}, \quad \|\nabla\tilde\phi\|_{H^m} \le C\|\tilde h\|_{H^{m-1}} \le C\|\tilde h\|_{H^m},
\]
we get
\begin{align}\label{eq:J1-est}
\begin{aligned}
|{\mathsf J}_1^\alpha| &\le \frac{C}{\nu} \|\nabla\pa^\alpha w\|_{L^2} \lt( \|h\|_{H^m} + \|\tilde h\|_{H^m} \rt) \|w\|_{H^m} \cr
&\le \frac{1}{16\nu}\|\nabla\pa^\alpha w\|_{L^2}^2 + \frac{C}{\nu}\lt( \|h\|_{H^m}^2+\|\tilde h\|_{H^m}^2 \rt) \|w\|_{H^m}^2.
\end{aligned}
\end{align}

For ${\mathsf J}_2^\alpha$, we integrate by parts once:
\[
{\mathsf J}_2^\alpha = -\frac1\nu \intt \nabla\pa^\alpha w\cdot \pa^\alpha\nabla\cdot(\rho u\otimes u)\,dx.
\]
By the product estimate,
\[
\|\pa^\alpha\nabla\cdot(\rho u\otimes u)\|_{L^2} \le C(1+\|h\|_{H^{m+1}})\|u\|_{H^{m+1}}^2.
\]
Thus, using the bootstrap bound, this gives
\[
|{\mathsf J}_2^\alpha| \le \frac{C}{\nu} (\|h\|_{H^{m+1}} + \|\tilde h\|_{H^{m+1}}) (1+\|h\|_{H^{m+1}}) \|u\|_{H^{m+1}}^2 \le  \frac{C(1+\calM_T)^2}{\nu}\|u\|_{H^{m+1}}^2,
\]
where $C>0$ depends on the auxiliary bound $\tilde C_{m+1}$.

Next, by definition,
\[
{\mathsf J}_3^\alpha = \frac1\nu\|\pa^\alpha\nabla\cdot(\rho u)\|_{L^2}^2.
\]
The product estimate gives
\[
\|\pa^\alpha\nabla\cdot(\rho u)\|_{L^2} \le C(1+\|h\|_{H^{m+1}})\|u\|_{H^{m+1}}.
\]
Hence
\[
{\mathsf J}_3^\alpha \le \frac{C(1+\calM_T)^2}{\nu}\|u\|_{H^{m+1}}^2.
\]

Finally, using the bootstrap bound, \eqref{eq:tilde-bdd}, and Young's inequality, and absorbing the fixed factor $1+\|\tilde h\|_{H^{m+1}}^2\le 1+\tilde C_{m+1}$ into $C$, we obtain
\begin{align*}
|{\mathsf J}_4^\alpha| &\le \frac1{\nu^2} \| \pa^\alpha\nabla\cdot(\rho u)\|_{L^2} \lt\|  \pa^\alpha\nabla\cdot \lt( \tilde\rho\nabla\tilde\phi+\nabla\tilde h \rt) \rt\|_{L^2}\cr
&\le \frac{C}{\nu^2} (1+\|h\|_{H^{m+1}})\|u\|_{H^{m+1}} (1+\|\tilde h\|_{H^{m+1}}) \|\tilde h\|_{H^{m+2}} \cr
&\le \frac{C(1+\calM_T)^2}{\nu} \|u\|_{H^{m+1}}^2 + \frac{C}{\nu^3} \|\tilde h\|_{H^{m+2}}^2.
\end{align*}

Combining the above estimates, summing over $|\alpha|=m$, and absorbing the small portion of $\|w\|_{H^{m+1}}^2$ from \eqref{eq:J1-est}, we obtain
\bq\label{eq:Ew-diff}
\frac{d}{dt}\calE_m^w(t) + \frac{1}{2\nu}\|w\|_{H^{m+1}}^2 \le \frac{C}{\nu} \lt( \|h\|_{H^m}^2+\|\tilde h\|_{H^m}^2 \rt) \|w\|_{H^m}^2 + \frac{C(1+\calM_T)^2}{\nu} \|u\|_{H^{m+1}}^2 + \frac{C}{\nu^3} \|\tilde h\|_{H^{m+2}}^2 .
\eq

We next compare $\calE_m^w$ with $\|w\|_{H^m}^2$. Since $w$ has zero spatial mean, the homogeneous $H^m$-seminorm is equivalent to the full $H^m$-norm. By Young's inequality and the product estimate,
\[
\lt| \frac2\nu \sum_{|\alpha|=m} \intt   \pa^\alpha w\, \pa^\alpha R\,dx \rt| \le \frac12\|\nabla^m w\|_{L^2}^2 + \frac{C}{\nu^2}\|R\|_{H^m}^2  \le \frac12\|\nabla^m w\|_{L^2}^2 + \frac{C(1+\calM_T)^2}{\nu^2} \|u\|_{H^{m+1}}^2.
\]
Consequently, using the definition of $\calM_T$, we have
\bq\label{eq:w-Ew}
\|w\|_{H^m}^2 \le C\calE_m^w(t) + \frac{C(1+\calM_T)^2\calM_T^2}{\nu^2}.
\eq
Using \eqref{eq:w-Ew} in \eqref{eq:Ew-diff}, we get
\begin{align}\label{eq:Ew-gro0}
\begin{aligned}
\frac{d}{dt}\calE_m^w(t) + \frac{1}{2\nu}\|w\|_{H^{m+1}}^2 &\le \frac{C}{\nu} \lt( \|h\|_{H^m}^2+\|\tilde h\|_{H^m}^2 \rt) \calE_m^w(t) + \frac{C(1+\calM_T)^2\calM_T^2}{\nu^3} \lt( \|h\|_{H^m}^2+\|\tilde h\|_{H^m}^2 \rt) \cr
&\quad + \frac{C(1+\calM_T)^2}{\nu} \|u\|_{H^{m+1}}^2 + \frac{C}{\nu^3} \|\tilde h\|_{H^{m+2}}^2 .
\end{aligned}
\end{align}
 
Since $h(0)=\tilde h(0)$, we have $w(0)=0$, and hence
\[
\calE_m^w(0)=0.
\]
Moreover, by \eqref{eq:den-bdd1} and \eqref{eq:tilde-bdd},
\[
\frac1\nu\int_0^T \lt( \|h(s)\|_{H^m}^2+\|\tilde h(s)\|_{H^{m+2}}^2 \rt)\,ds \le C.
\]
Here $C$ depends only on the initial data, $d,m$, and the density bounds, and is independent of $T$, $\nu$, and $\calM_T$.  Applying Gronwall's lemma to \eqref{eq:Ew-gro0} and using the definition of $\calM_T$, we find
\begin{align*}
\calE_m^w(t) + \frac1\nu\int_0^t\|w(s)\|_{H^{m+1}}^2\,ds &\le \frac{C(1+\calM_T)^2}{\nu} \int_0^t\|u(s)\|_{H^{m+1}}^2\,ds \cr
&\quad + \frac{C(1+\calM_T)^2\calM_T^2}{\nu^3} \int_0^t \lt( \|h(s)\|_{H^m}^2+\|\tilde h(s)\|_{H^m}^2 \rt)\,ds \cr
&\quad + \frac{C}{\nu^3} \int_0^t \|\tilde h(s)\|_{H^{m+2}}^2\,ds \cr
&\le \frac{C(1+\calM_T)^2(1+\calM_T^2)}{\nu^2}.
\end{align*}
Using \eqref{eq:w-Ew} once more, we conclude that
\[
\sup_{0\le t\le T}\|w(t)\|_{H^m}^2 + \frac1\nu\int_0^T\|w(t)\|_{H^{m+1}}^2\,dt \le \frac{C(1+\calM_T)^2(1+\calM_T^2)}{\nu^2}.
\]
This completes the proof.
\end{proof}

%
%
%
%
%

\section{Proof of Theorem \ref{thm:main}}\label{sec:proof-main}

In this section, we prove Theorem \ref{thm:main}. The proof is divided into two parts. First, we close the bootstrap argument and obtain global-in-time existence of smooth solutions. Second, we prove the exponential relaxation estimate.

%
%
%
%
%

\subsection{Global existence}\label{ssec:gwp}

Let
\[
h_0:=\rho_0-1.
\]
Let $\tilde C_{m+1}$ be the constant from Proposition \ref{prop:aux-sob}. Then, we get
\[
\|\tilde h(t)\|_{H^{m+1}}^2 + \frac1\nu \int_0^t \|\tilde h(s)\|_{H^{m+2}}^2\,ds \le \tilde C_{m+1}, \quad t\ge0.
\]
Moreover, by Lemma \ref{lem:aux-maximum}, 
\bq\label{eq:aux-maxip}
\rho_1\le \tilde\rho(t,x)\le \rho_2,
\quad t\ge0,\quad x\in\T^d.
\eq
All constants below may depend on $d,m,\rho_1,\rho_2$ and the initial data, but are independent of $T$ and $\nu$.

\medskip
\noindent{\bf Bootstrap setting.}
Let $T>0$ be such that a smooth solution exists on $[0,T]$. We recall the bootstrap quantity
\[
\calM_T := \sup_{0\le t\le T} \lt( \|h(t)\|_{H^{m+1}}+\|u(t)\|_{H^{m+1}} \rt) + \lt( \nu\int_0^T\|u(t)\|_{H^{m+1}}^2\,dt \rt)^{1/2}.
\]
We introduce a bootstrap radius $R>0$, to be fixed after the a priori
estimates are applied. Assume that, on $[0,T]$,
\bq\label{eq:M-radi}
\calM_T\le R,
\eq
and
\bq\label{eq:b-den-bd}
\frac{\rho_1}{2}
\le
1+h(t,x)
\le
2\rho_2,
\quad (t,x)\in[0,T]\times\T^d.
\eq
We also assume the large damping bootstrap condition
\bq\label{eq:b-small}
\frac{\calM_T(1+\calM_T)}{\nu}
\le
\delta_0,
\eq
where $\delta_0>0$ is the small constant fixed in the previous sections. Once $R$ is chosen in terms of the initial data and the density bounds, the last condition will be ensured by taking $\nu$ sufficiently large.

\medskip
\noindent{\bf Closing the Sobolev bootstrap.}
Under \eqref{eq:b-den-bd} and \eqref{eq:b-small}, Proposition \ref{prop:mod-den} gives a constant $C_1>0$, independent of $T$ and $\nu$, such that 
\[
\sup_{0\le t\le T}\|h(t)\|_{H^m}^2 + \frac1\nu \int_0^T \|h(t)\|_{H^{m+1}}^2\,dt \le C_1.
\]
Using this estimate in Proposition \ref{prop:top-order}, we obtain
\bq\label{eq:t-bd}
\sup_{0\le t\le T} \lt( \|h(t)\|_{H^{m+1}}^2+\|u(t)\|_{H^{m+1}}^2 \rt) + \nu\int_0^T\|u(t)\|_{H^{m+1}}^2\,dt \le e^{CC_1}\lt( \|h_0\|_{H^{m+1}}^2+\|u_0\|_{H^{m+1}}^2 \rt).
\eq
We now determine the bootstrap radius. Set
\[
C_2^* := e^{CC_1} \lt( \|h_0\|_{H^{m+1}}^2+\|u_0\|_{H^{m+1}}^2 \rt),
\]
and choose
\[
R:=6\sqrt{C_2^*}.
\]
Then \eqref{eq:t-bd} implies
\[
\sup_{0\le t\le T} \lt( \|h(t)\|_{H^{m+1}}+\|u(t)\|_{H^{m+1}} \rt) \le \sqrt2\,\sqrt{C_2^*},
\]
and
\[
\lt( \nu\int_0^T\|u(t)\|_{H^{m+1}}^2\,dt \rt)^{1/2} \le \sqrt{C_2^*}.
\]
Thus, we have
\bq\label{eq:M-R}
\calM_T \le (\sqrt2+1)\sqrt{C_2^*} \le 3\sqrt{C_2^*} = \frac R2.
\eq
Hence, the Sobolev bootstrap \eqref{eq:M-radi} is improved. Moreover, if
\[
\frac{R(1+R)}{\nu} \le \frac{\delta_0}{2},
\]
then \eqref{eq:M-R} also gives
\[
\frac{\calM_T(1+\calM_T)}{\nu} \le \frac{\delta_0}{2}.
\]
Therefore, the large damping smallness condition is improved as well.

\medskip
\noindent{\bf Closing the density bootstrap.}
It remains to improve the pointwise density bound. Let
\[
w:=h-\tilde h.
\]
By Proposition \ref{prop:comp-den} and \eqref{eq:M-R},
\[
\sup_{0\le t\le T}\|w(t)\|_{H^m}^2 \le \frac{C(1+\calM_T)^2(1+\calM_T^2)}{\nu^2} \le \frac{C(1+R)^2(1+R^2)}{\nu^2}.
\]
Thus, for some constant $C_*>0$ depending only on the initial data, $d,m$, and the density bounds,
\[
\sup_{0\le t\le T}\|w(t)\|_{H^m} \le \frac{C_*}{\nu}.
\]
Since $m>d/2$, Sobolev embedding gives
\[
\|w(t)\|_{L^\infty} \le C_0\|w(t)\|_{H^m}
\]
for some $C_0>0$. Choose $\nu$ sufficiently large so that
\[
\frac{C_0C_*}{\nu} \le \frac12\rho_1.
\]
Then
\[
\sup_{0\le t\le T}\|w(t)\|_{L^\infty} \le \frac12\rho_1.
\]
Combining this with \eqref{eq:aux-maxip}, we obtain
\[
1+h(t,x) = \tilde\rho(t,x)+w(t,x) \ge \rho_1-\frac{\rho_1}{2} = \frac{\rho_1}{2},
\]
and
\[
1+h(t,x) = \tilde\rho(t,x)+w(t,x) \le \rho_2+\frac{\rho_1}{2} \le 2\rho_2.
\]
Thus the density bootstrap bound \eqref{eq:b-den-bd} is closed.

\medskip
\noindent{\bf Choice of damping threshold and continuation.} We now choose $\nu_0>0$ sufficiently large so that
\[
\frac{R(1+R)}{\nu_0} \le \frac{\delta_0}{2} \quad \text{and} \quad \frac{C_0C_*}{\nu_0} \le \frac12\rho_1.
\]
Then, for every $\nu\ge\nu_0$, the Sobolev and large-damping bootstrap bounds
are improved, while the density bootstrap bound is closed:
\[
\calM_T\le \frac R2, \quad \frac{\calM_T(1+\calM_T)}{\nu}\le\frac{\delta_0}{2}, \quad \frac{\rho_1}{2}\le 1+h\le 2\rho_2.
\]
 
By the local well-posedness and continuation criterion in Theorem \ref{thm:lwp}, there exists a smooth solution on a short time interval, since
\[
\rho_1\le \rho_0(x)\le \rho_2.
\]
Moreover, by the choice of $R$ and $\nu_0$, the bootstrap assumptions hold initially. The estimates above close the bootstrap bounds on any time interval on which the smooth solution exists.

In particular, the $H^{m+1}$ norm of $(h,u)$ remains bounded uniformly in time and the density remains strictly positive. Therefore the continuation criterion in Theorem \ref{thm:lwp} implies that the local solution extends globally in time. Consequently, for every $\nu\ge\nu_0$, the system \eqref{eq:main}--\eqref{eq:initial} admits a global-in-time smooth solution satisfying
\[
(\rho,u) \in C([0,\infty);H^{m+1}(\T^d)) \times C([0,\infty);H^{m+1}(\T^d)),
\]
and
\[
\rho(t,x)>0, \quad t\ge0,\quad x\in\T^d.
\]
This proves the global existence part of Theorem \ref{thm:main}.

%
%
%
%
%
 
\subsection{Large-time behavior}

We now prove the exponential relaxation estimate in Theorem \ref{thm:main}. Let $(\rho,u,\phi)$ be the global smooth solution constructed above. From the global existence argument, we have uniform pointwise density bounds
\[
0<\underline\rho\le \rho(t,x)\le \overline\rho<\infty, \quad (t,x)\in[0,\infty)\times\T^d,
\]
and a uniform Sobolev bound
\[
\sup_{t\ge0} \lt( \|\rho(t)-1\|_{H^{m+1}} + \|u(t)\|_{H^{m+1}} \rt) \le C_*.
\]
The constants in this subsection may depend on $d,m,\underline\rho,\overline\rho,C_*$, but are independent of $t$ and $\nu$ for $\nu\ge\nu_0$.

We consider the free energy
\[
\calE(t) = \frac12\intt \rho |u|^2\,dx + \intt (\rho\log\rho-\rho+1)\,dx + \frac12\intt |\nabla\phi|^2\,dx.
\]
Since $\rho$ is bounded above and below, the entropy density $\rho\log\rho-\rho+1$ is equivalent to $|\rho-1|^2$. Hence, there exist constants $c_0, C_0 > 0$ such that
\bq\label{eq:fe-equi}
c_0\lt(\|u\|_{L^2(\rho)}^2 + \|\rho-1\|_{L^2}^2 + \|\nabla\phi\|_{L^2}^2 \rt) \le \calE(t) \le C_0 \lt(\|u\|_{L^2(\rho)}^2 + \|\rho-1\|_{L^2}^2 + \|\nabla\phi\|_{L^2}^2\rt).
\eq

We also recall from Lemma \ref{lem:be} that the free energy $\calE$ satisfies
\bq\label{eq:fe-iden}
\frac{d}{dt}\calE(t) + \nu\intt \rho|u|^2\,dx = 0.
\eq

The identity \eqref{eq:fe-iden} dissipates only the velocity. To recover the density and potential dissipation, we introduce the interaction functional
\[
\calI(t) := \intt \rho u\cdot\nabla\phi\,dx.
\]
Using
\[
\pa_t(\rho u) = -\nabla\cdot(\rho u\otimes u) -\nabla\rho -\rho\nabla\phi -\nu\rho u,
\]
we compute
\begin{align}\label{eq:I-deri}
\begin{aligned}
\frac{d}{dt}\calI(t) &= \intt  \rho u\otimes u:\nabla^2\phi\,dx - \intt  \nabla\rho\cdot\nabla\phi\,dx - \intt  \rho|\nabla\phi|^2\,dx\cr
&\quad -\nu\intt \rho u\cdot\nabla\phi\,dx + \intt \rho u\cdot\nabla\phi_t\,dx.
\end{aligned}
\end{align}
We estimate the terms on the right-hand side. Since $-\Delta\phi=\rho-1$ and $\intt (\rho-1)\,dx=0$, we get
\[
-\intt \nabla\rho\cdot\nabla\phi\,dx = -\|\rho-1\|_{L^2}^2.
\]
Moreover,
\[
-\intt \rho|\nabla\phi|^2\,dx \le -\underline\rho\|\nabla\phi\|_{L^2}^2.
\]
For the convective term, using the uniform Sobolev bound and elliptic regularity,
\[
\|\nabla^2\phi\|_{L^2}\le C\|\rho-1\|_{L^2},
\]
we get, for any $\eta>0$,
\[
\lt| \intt  \rho u\otimes u:\nabla^2\phi\,dx \rt| \le C\|u\|_{L^\infty}\lt(\intt \rho|u|^2\,dx\rt)^{1/2} \|\rho-1\|_{L^2} \le \eta\|\rho-1\|_{L^2}^2 + C_\eta \intt \rho|u|^2\,dx.
\]
Next, differentiating the Poisson equation in time gives
\[
-\Delta\phi_t=\rho_t=-\nabla\cdot(\rho u).
\]
Thus,
\[
\|\nabla\phi_t\|_{L^2} \le C\|\rho u\|_{L^2} \le C \lt( \intt \rho|u|^2\,dx \rt)^{1/2},
\]
and this deduces
\bq\label{eq:phit-I}
\lt|
\intt \rho u\cdot\nabla\phi_t\,dx
\rt|
\le
C\intt \rho|u|^2\,dx.
\eq
Combining \eqref{eq:I-deri}--\eqref{eq:phit-I}, and choosing $\eta>0$ sufficiently small, we obtain
\bq\label{eq:I-deri1}
\frac{d}{dt}\calI(t) \le -c_1 \lt( \|\rho-1\|_{L^2}^2 + \|\nabla\phi\|_{L^2}^2 \rt) + C_1 \intt \rho|u|^2\,dx -\nu\calI(t).
\eq

We also record that, for any $\eta>0$,
\bq\label{eq:I-bound-LT}
|\calI(t)| \le \lt(\intt \rho|u|^2\,dx\rt)^{1/2} \lt(\intt \rho|\nabla\phi|^2\,dx\rt)^{1/2} \le \eta\|\nabla\phi\|_{L^2}^2 + C_\eta\intt \rho|u|^2\,dx.
\eq

We now define the modified Lyapunov functional
\[
\widetilde{\calE}(t) := \calE(t) + \frac{\delta}{\nu}\calI(t),
\]
where $\delta>0$ will be chosen sufficiently small. By \eqref{eq:I-bound-LT}, for $\delta>0$ sufficiently small and $\nu\ge1$,
\[
\frac12\calE(t) \le \widetilde{\calE}(t) \le \frac32\calE(t).
\]

Using \eqref{eq:fe-iden} and \eqref{eq:I-deri1}, we compute
\[
\frac{d}{dt}\widetilde{\calE}(t) \le -\nu\intt \rho|u|^2\,dx - \frac{\delta c_1}{\nu} \lt( \|\rho-1\|_{L^2}^2 + \|\nabla\phi\|_{L^2}^2 \rt)  + \frac{\delta C_1}{\nu}
\intt \rho|u|^2\,dx -\delta\calI(t).
\]
The last term is controlled by \eqref{eq:I-bound-LT} with the scaling adapted to the two dissipative rates:
\[
\delta|\calI(t)| \le \frac{\delta c_1}{2\nu} \lt( \|\rho-1\|_{L^2}^2 + \|\nabla\phi\|_{L^2}^2 \rt) + C\delta\nu \intt \rho|u|^2\,dx.
\]
Choose $\delta>0$ sufficiently small, independently of $\nu$, so that $C\delta\le 1/4$. The term $\delta C_1/\nu$ is also absorbed by the velocity dissipation after taking $\delta$ small enough. Thus, we obtain
\bq\label{eq:mod-dec}
\frac{d}{dt}\widetilde{\calE}(t) + c_2\nu\intt \rho|u|^2\,dx + \frac{c_2}{\nu} \lt( \|\rho-1\|_{L^2}^2 + \|\nabla\phi\|_{L^2}^2 \rt) \le0.
\eq
Here $c_2>0$ is independent of $t$ and $\nu$.

By the equivalence of $\widetilde{\calE}$ and $\calE$, together with \eqref{eq:fe-equi}, the dissipation in \eqref{eq:mod-dec} controls $\widetilde{\calE}$ with rate
\[
\theta_0=c\min\lt\{\nu,\frac1\nu\rt\}.
\]
In particular, for $\nu\ge1$, one may take
\[
\theta_0=\frac{c}{\nu}.
\]
Thus, we obtain
\[
\frac{d}{dt}\widetilde{\calE}(t) + \theta_0\widetilde{\calE}(t) \le0,
\]
and applying Gr\"onwall's lemma to yield
\[
\widetilde{\calE}(t) \le \widetilde{\calE}(0)e^{-\theta_0 t}.
\]
Using once more the equivalence between $\widetilde{\calE}$ and
$\calE$, we conclude that
\bq\label{eq:lt_fin}
\intt \rho(t)|u(t)|^2\,dx + \|\rho(t)-1\|_{L^2}^2 + \|\nabla\phi(t)\|_{L^2}^2  \le C \lt( \intt \rho_0|u_0|^2\,dx + \|\rho_0-1\|_{L^2}^2 \rt)e^{-\theta_0 t}.
\eq

We now establish the decay estimate in $H^{m+1}$. Since the density remains bounded below and above uniformly in time, \eqref{eq:lt_fin} implies
\bq\label{eq:L2-decay}
    \|u(t)\|_{L^2}^2+\|h(t)\|_{L^2}^2 \le C e^{-\theta_0 t}, \quad h:=\rho-1.
\eq
On the other hand, the global estimate obtained in Section \ref{ssec:gwp} gives
\bq\label{eq:uniform-Hm1}
\sup_{t\ge0} \lt( \|\rho(t)-1\|_{H^{m+1}} + \|u(t)\|_{H^{m+1}} \rt) \le C.
\eq
Interpolating between \eqref{eq:L2-decay} and \eqref{eq:uniform-Hm1}, we obtain
\bq\label{eq:Hm-decay-aux}
    \|h(t)\|_{H^m}^2+\|u(t)\|_{H^m}^2 \le C e^{-\frac{\theta_0}{m+1}t}
\eq
for some $C>0$ independent of $t$. We use this estimate only as an auxiliary bound in the top-order argument below.

We first recall the modified density energy introduced in the proof of Proposition \ref{prop:mod-den}:
\[
\calE_m^h(t) := \|\nabla^m h(t)\|_{L^2}^2 - \frac{2}{\nu} \sum_{|\alpha|=m} \intt \pa^\alpha h\, \pa^\alpha\nabla\cdot((1+h)u)\,dx.
\]
Taking $k=m$ in \eqref{est_h_high}, we have
\bq\label{eq:top-density-decay}
\frac{d}{dt}\calE_m^h + \frac{2}{\nu}\|\nabla^m h\|_{H^1}^2 \le \frac{C}{\nu} \|\nabla^{m+1}h\|_{L^2} \|\nabla^m h\|_{L^2}^2 +\frac{C}{\nu}\|u\|_{H^{m+1}}^2.
\eq
Here and below, the constants may depend on the initial data and the uniform density bounds, but are independent of $t$ and $\nu$. Using \eqref{eq:uniform-Hm1}, \eqref{eq:Hm-decay-aux}, and
\[
\|u\|_{H^{m+1}}^2 \le C\lt(  \|u\|_{L^2}^2 + \|\nabla^{m+1}u\|_{L^2}^2\rt),
\]
we infer from \eqref{eq:top-density-decay} that
\bq\label{eq:top-density-decay2}
    \frac{d}{dt}\calE_m^h + \frac{c}{\nu}\|\nabla^{m+1}h\|_{L^2}^2 \le C e^{-\frac{\theta_0}{m+1}t} +\frac{C}{\nu}\|\nabla^{m+1}u\|_{L^2}^2.
\eq

We next use the top-order hyperbolic energy $\calE_{m+1}$ defined in the proof of Proposition \ref{prop:top-order}. By the uniform $H^{m+1}$ bound and Sobolev embedding, we may increase the damping threshold $\nu_0$ if necessary so that the first term on the right-hand side of \eqref{est_u_high} can be absorbed into the velocity dissipation. Applying Young's inequality to the remaining mixed term, we obtain
\[\begin{aligned}
\frac12&\frac{d}{dt}\mathcal{E}_{m+1} + \nu \|\nabla^{m+1} u\|_{L^2}^2 \\
&\le  C\|\nabla u\|_{L^\infty}\|\nabla^{m+1} u\|_{L^2}^2+ C\lt(\|\nabla u\|_{L^\infty} + \|\nabla h\|_{L^\infty}\rt)\|\nabla^{m+1} u\|_{L^2}\|\nabla^{m+1} h\|_{L^2}\\
&\le \frac\nu4\|\nabla^{m+1} u\|_{L^2}^2 + C e^{-\frac{\theta_0(2m-d)}{2(m+1)}t}.
\end{aligned}\]
Here, in the last inequality, we used the uniform $H^{m+1}$ bound together with the decay estimates for the $W^{1,\infty}$ norms. Indeed, by the Gagliardo--Nirenberg inequality and \eqref{eq:L2-decay}, we have
\[
\|\nabla u(t)\|_{L^\infty} \le C\|\nabla^{m+1} u(t)\|_{L^2}^{\frac{d+2}{2(m+1)}} \|u(t)\|_{L^2}^{\frac{2m-d}{2(m+1)}} \le C e^{-\frac{\theta_0(2m-d)}{4(m+1)}t},
\]
and similarly,
\[
\|\nabla h(t)\|_{L^\infty} \le C e^{-\frac{\theta_0(2m-d)}{4(m+1)}t}.
\]
Consequently, we obtain
\bq\label{eq:top-hyp-decay2}
\frac{d}{dt}\calE_{m+1} + \nu\|\nabla^{m+1}u\|_{L^2}^2 \le C e^{-\frac{\theta_0(2m-d)}{2(m+1)}t}.
\eq

We now fix $C^*>1$ sufficiently large and, if necessary, increase $\nu_0$ further so that $\frac{C^*\nu}{2} > \frac{C}{\nu}$ for all $\nu \ge \nu_0$. Define
\[
\calH_{m+1} := \calE_m^h+C^*\calE_{m+1}.
\]
By the equivalence properties of the modified density energy and the top-order hyperbolic energy established above, there exists $c_1>1$ such that
\bq\label{eq:comb-ener-equi}
\frac{1}{c_1} \lt(  \|\nabla h\|_{H^m}^2 +\|\nabla u\|_{H^m}^2 \rt) \le \calH_{m+1} \le c_1 \lt( \|\nabla h\|_{H^m}^2 +\|\nabla u\|_{H^m}^2\rt).
\eq
Combining \eqref{eq:top-density-decay2} and $C^*$ times \eqref{eq:top-hyp-decay2}, we obtain
\[
\frac{d}{dt}\calH_{m+1} + \frac{c}{\nu}\|\nabla^{m+1}h\|_{L^2}^2 + c\nu\|\nabla^{m+1}u\|_{L^2}^2 \le C e^{-\theta_1t}, \quad \theta_1 := \theta_0 \frac{\min\{1, (m-\frac d2)\}}{m+1}.
\]
Since each component of $\nabla h$ and $\nabla u$ has zero mean, Poincar\'e's inequality gives
\[
\|\nabla h\|_{H^m}^2 \le C\|\nabla^{m+1}h\|_{L^2}^2, \quad \|\nabla u\|_{H^m}^2 \le C\|\nabla^{m+1}u\|_{L^2}^2.
\]
Using \eqref{eq:comb-ener-equi}, we thus deduce that
\[
\frac{d}{dt}\calH_{m+1} + c\min\lt\{ \nu,\frac{1}{\nu} \rt\} \calH_{m+1} \le C e^{-\theta_1t}.
\]
Since $\theta_0$ is of order $\min\{\nu,\nu^{-1}\}$, Gr\"onwall's lemma yields
\[
\calH_{m+1}(t) \le C e^{-c\min\{\nu,\nu^{-1}\}t}.
\]
Consequently, we have
\[
\|\nabla h(t)\|_{H^m}^2 +\|\nabla u(t)\|_{H^m}^2 \le C e^{-c\min\{\nu,\nu^{-1}\}t}.
\]
Combining this estimate with \eqref{eq:L2-decay} concludes the desired exponential relaxation estimate. Together with the global existence result obtained in Section \ref{ssec:gwp}, this completes the proof of Theorem \ref{thm:main}.

%
%
%
%
%

\section{Proof of Theorem \ref{thm:main2}}\label{sec:overdamped-limit}

In this section, we prove Theorem \ref{thm:main2}. We keep the parameter $\nu$ explicit and work in the slow time variable $s=t/\nu$.

For each $\nu\ge\nu_0$, let $(\rho^\nu,u^\nu,\phi^\nu)$ be the global smooth solution constructed in Theorem \ref{thm:main}. We define the slow-time
density, potential, and rescaled flux by
\[
\rho_\nu(s,x):=\rho^\nu(\nu s,x), \quad \Phi_\nu(s,x):=\phi^\nu(\nu s,x), \quad \text{and} \quad J_\nu(s,x):=\nu \rho^\nu u^\nu(\nu s,x).
\]
Then
\[
 \pa_s\rho_\nu+\nabla\cdot J_\nu=0, \quad -\Delta\Phi_\nu=\rho_\nu-1.
\]
Moreover, the momentum equation becomes
\bq\label{eq:mom-nu}
\frac1{\nu^2} \pa_sJ_\nu + \frac1{\nu^2}\nabla\cdot \lt( \frac{J_\nu\otimes J_\nu}{\rho_\nu} \rt) + \nabla\rho_\nu + \rho_\nu\nabla\Phi_\nu = -J_\nu.
\eq

The formal limit is the drift-diffusion--Poisson system
\bq\label{eq:ddp}
 \pa_s\bar\rho -\nabla\cdot(\bar\rho\nabla\bar\Phi) = \Delta\bar\rho, \quad -\Delta\bar\Phi=\bar\rho-1, \quad \bar\rho(0,x)=\rho_0(x),
\eq
and the corresponding drift-diffusion flux is
\[
\bar J:=-\nabla\bar\rho-\bar\rho\nabla\bar\Phi.
\]

%
%
%
%
%

\subsection{Density error}

We first derive the density error estimate. Let $\tilde\rho$ be the auxiliary density introduced in Section \ref{sec:aux-comp}.
Since $\tilde\rho$ solves
\[
\pa_t\tilde\rho -\frac1\nu\nabla\cdot(\tilde\rho\nabla\tilde\phi) = \frac1\nu\Delta\tilde\rho, \quad -\Delta\tilde\phi=\tilde\rho-1,
\]
the rescaled function
\[
\bar\rho(s,x):=\tilde\rho(\nu s,x), \quad \bar\Phi(s,x):=\tilde\phi(\nu s,x),
\]
solves \eqref{eq:ddp}. By uniqueness, it coincides with the limit solution of \eqref{eq:ddp}. Thus, we have
\[
\rho_\nu(s)-\bar\rho(s) = \rho^\nu(\nu s)-\tilde\rho(\nu s).
\]
Using Proposition \ref{prop:comp-den} on $[0,\infty)$, we obtain
\bq\label{eq:od-den-nu}
\sup_{s\ge0} \|\rho_\nu(s)-\bar\rho(s)\|_{H^m}^2 \le \frac{C}{\nu^2}.
\eq
Moreover, the change of variables $t=\nu s$ gives
\[
\int_0^\infty \|\rho_\nu(s)-\bar\rho(s)\|_{H^{m+1}}^2\,ds = \frac1\nu \int_0^\infty \|\rho^\nu(t)-\tilde\rho(t)\|_{H^{m+1}}^2\,dt.
\]
The comparison estimate in Proposition \ref{prop:comp-den} gives
\[
\frac1\nu \int_0^\infty \|\rho^\nu(t)-\tilde\rho(t)\|_{H^{m+1}}^2\,dt \le \frac{C}{\nu^2},
\]
and consequently, 
\bq\label{eq:od-den-nu2}
\int_0^\infty
\|\rho_\nu(s)-\bar\rho(s)\|_{H^{m+1}}^2\,ds \le \frac{C}{\nu^2}.
\eq

%
%
%
%
%

\subsection{Initial layer and flux defect}

Since
\[
J_\nu(0,x)=\nu\rho_0(x)u_0(x),
\]
the rescaled flux contains a fast initial layer. Following the standard relaxation-limit strategy, we remove this layer by introducing $J_L^\nu$ as the solution to the damped heat equation
\bq\label{eq:JL-def}
\frac1{\nu^2} \lt(  \pa_sJ_L^\nu-\Delta J_L^\nu \rt) + J_L^\nu = 0, \quad J_L^\nu(0,x)=\nu\rho_0(x)u_0(x).
\eq
We set
\[
\widehat J_\nu:=J_\nu-J_L^\nu
\]
and define the flux defect by
\[
D_\nu := \widehat J_\nu + \nabla\rho_\nu + \rho_\nu\nabla\Phi_\nu.
\]
Thus $D_\nu$ measures the error in the drift-diffusion flux relation after subtracting the fast initial layer.

We first provide the estimates for $J_L^\nu$.

\begin{lemma} \label{lem:JL-nu}
Let $J_L^\nu$ solve \eqref{eq:JL-def}. Then, for any integer $r\ge0$, we have
\bq\label{eq:JL-est0}
\frac1{\nu^2} \sup_{s\ge0} \|J_L^\nu(s)\|_{H^r}^2 + \int_0^\infty \|J_L^\nu(s)\|_{H^r}^2\,ds + \frac1{\nu^2} \int_0^\infty \|\nabla J_L^\nu(s)\|_{H^r}^2\,ds \le C\|\rho_0u_0\|_{H^r}^2.
\eq
In particular, if $\rho_0,u_0\in H^{m+1}$, then
\bq\label{eq:JL-lap-est}
\frac1{\nu^4} \int_0^\infty \|\Delta J_L^\nu(s)\|_{H^{m-1}}^2\,ds \le \frac{C}{\nu^2}.
\eq
\end{lemma}

\begin{proof}
Apply $ \pa^\alpha$ to \eqref{eq:JL-def}, multiply by $ \pa^\alpha J_L^\nu$, integrate over $\T^d$, and sum over $|\alpha|\le r$. This gives
\[
\frac1{2\nu^2} \frac{d}{ds} \|J_L^\nu\|_{H^r}^2 + \frac1{\nu^2} \|\nabla J_L^\nu\|_{H^r}^2 + \|J_L^\nu\|_{H^r}^2 =0.
\]
Integrating in $s$ and using
\[
J_L^\nu(0)=\nu\rho_0u_0
\]
yields \eqref{eq:JL-est0}. Taking $r=m+1$ in \eqref{eq:JL-est0} and using
\[
\|\Delta J_L^\nu\|_{H^{m-1}}\le C\|J_L^\nu\|_{H^{m+1}}
\]
gives \eqref{eq:JL-lap-est}.
\end{proof}

We next collect the slow-time bounds needed for the flux defect estimate.

\begin{lemma} \label{lem:st-bds-nu}
Under the assumptions of Theorem \ref{thm:main}, the slow-time variables satisfy
\[
\int_0^\infty \|J_\nu(s)\|_{H^{m+1}}^2\,ds + \int_0^\infty \|\rho_\nu(s)-1\|_{H^{m+1}}^2\,ds \le C.
\]
Define
\[
A_\nu:=\nabla\rho_\nu+\rho_\nu\nabla\Phi_\nu.
\]
Then we have
\bq\label{eq:A-bd}
\int_0^\infty \|A_\nu(s)\|_{H^{m-1}}^2\,ds + \int_0^\infty \| \pa_sA_\nu(s)\|_{H^{m-1}}^2\,ds \le C.
\eq
Moreover,
\[
\frac1{\nu^4} \int_0^\infty \lt\| \nabla\cdot \lt( \frac{J_\nu\otimes J_\nu}{\rho_\nu} \rt) \rt\|_{H^{m-1}}^2\,ds \le \frac{C}{\nu^2}.
\]
Here $C>0$ is independent of $\nu$.
\end{lemma}

\begin{proof}
Using the uniform upper and lower bounds of $\rho^\nu$, the global estimate
from Theorem \ref{thm:main}, and the change of variables $t=\nu s$, we get
\bq\label{eq:J_nu}
\int_0^\infty
\|J_\nu(s)\|_{H^{m+1}}^2\,ds \le C\nu^2 \int_0^\infty \|u^\nu(\nu s)\|_{H^{m+1}}^2\,ds = C\nu \int_0^\infty \|u^\nu(t)\|_{H^{m+1}}^2\,dt \le C.
\eq
Similarly,
\[
\int_0^\infty \|\rho_\nu(s)-1\|_{H^{m+1}}^2\,ds = \frac1\nu \int_0^\infty \|\rho^\nu(t)-1\|_{H^{m+1}}^2\,dt \le C,
\]
where we used the modified density estimate.

By elliptic regularity and the product estimate,
\[
\|A_\nu\|_{H^{m-1}} \le C\|\rho_\nu-1\|_{H^m}.
\]
This gives the first part of \eqref{eq:A-bd}. To estimate
$ \pa_sA_\nu$, we use
\[
 \pa_s\rho_\nu=-\nabla\cdot J_\nu, \quad -\Delta \pa_s\Phi_\nu= \pa_s\rho_\nu.
\]
Thus, by elliptic regularity and the product estimate,
\[
\| \pa_sA_\nu\|_{H^{m-1}} \le C\|J_\nu\|_{H^{m+1}}.
\]
Therefore \eqref{eq:A-bd} follows from the bound on $J_\nu$.

Finally, we estimate the convective flux term. Let
\[
u_\nu(s,x):=u^\nu(\nu s,x).
\]
Since $J_\nu=\nu\rho_\nu u_\nu$, we can write
\[
\frac{J_\nu\otimes J_\nu}{\rho_\nu} = \nu\, u_\nu\otimes J_\nu.
\]
Thus, using that $m>d/2$, we obtain
\[
\lt\| \nabla\cdot \lt( \frac{J_\nu\otimes J_\nu}{\rho_\nu} \rt) \rt\|_{H^{m-1}} \le C\nu \|u_\nu\otimes J_\nu\|_{H^m} \le C\nu \|u_\nu\|_{H^m} \|J_\nu\|_{H^m}.
\]
Hence, we have
\begin{align*}
\frac1{\nu^4} \int_0^\infty \lt\| \nabla\cdot \lt( \frac{J_\nu\otimes J_\nu}{\rho_\nu} \rt) \rt\|_{H^{m-1}}^2\,ds &\le \frac{C}{\nu^2} \int_0^\infty \|u_\nu(s)\|_{H^m}^2 \|J_\nu(s)\|_{H^m}^2\,ds \\
&\le \frac{C}{\nu^2} \lt( \sup_{s\ge0}\|u_\nu(s)\|_{H^m}^2 \rt) \int_0^\infty \|J_\nu(s)\|_{H^m}^2\,ds \\
&\le \frac{C}{\nu^2}.
\end{align*}
Here we used the uniform-in-time $H^{m+1}$ bound for $u^\nu$ and slow-time flux estimate \eqref{eq:J_nu}.
\end{proof}

We now prove the flux defect estimate.

\begin{proposition} 
Under the assumptions of Theorem \ref{thm:main}, there exists a constant
$C>0$, independent of $\nu$, such that
\bq\label{eq:Dar-def-est}
\int_0^\infty \|D_\nu(s)\|_{H^{m-1}}^2\,ds + \frac1{\nu^2} \sup_{s\ge0} \|\widehat J_\nu(s)\|_{H^{m-1}}^2 \le \frac{C}{\nu^2}.
\eq
\end{proposition}

\begin{proof}
Set
\[
\calN_\nu := \nabla\cdot \lt( \frac{J_\nu\otimes J_\nu}{\rho_\nu} \rt).
\]
From \eqref{eq:mom-nu}, we get
\[
\frac1{\nu^2} \pa_sJ_\nu + J_\nu + A_\nu = -\frac1{\nu^2}\calN_\nu, \quad A_\nu=\nabla\rho_\nu+\rho_\nu\nabla\Phi_\nu.
\]
On the other hand, \eqref{eq:JL-def} gives
\[
\frac1{\nu^2} \pa_sJ_L^\nu + J_L^\nu = \frac1{\nu^2}\Delta J_L^\nu.
\]
Subtracting the two equations and using
\[
\widehat J_\nu=J_\nu-J_L^\nu, \quad D_\nu=\widehat J_\nu+A_\nu,
\]
we obtain
\bq\label{eq:Jhat-eq}
\frac1{\nu^2} \pa_s\widehat J_\nu + D_\nu = -\frac1{\nu^2}\calN_\nu -\frac1{\nu^2}\Delta J_L^\nu.
\eq

For each multi-index $\alpha$ with $|\alpha|\le m-1$, we apply $\pa^\alpha$ to \eqref{eq:Jhat-eq}, multiply the resulting equation by $\pa^\alpha D_\nu$, and integrate over $\T^d$. Summing over
$|\alpha|\le m-1$, we obtain
\begin{align}\label{eq:defect-e0}
\begin{aligned}
\|D_\nu\|_{H^{m-1}}^2 &= -\frac1{\nu^2} \sum_{|\alpha|\le m-1} \intt \pa^\alpha\pa_s\widehat J_\nu \cdot \pa^\alpha D_\nu\,dx  -\frac1{\nu^2} \sum_{|\alpha|\le m-1} \intt \pa^\alpha\calN_\nu \cdot \pa^\alpha D_\nu\,dx \cr
&\quad -\frac1{\nu^2} \sum_{|\alpha|\le m-1} \intt \pa^\alpha\Delta J_L^\nu \cdot \pa^\alpha D_\nu\,dx .
\end{aligned}
\end{align}
Since $D_\nu=\widehat J_\nu+A_\nu$, we find
\begin{align}\label{eq:Jhat-t}
\begin{aligned}
-\frac1{\nu^2} \sum_{|\alpha|\le m-1} \intt \pa^\alpha\pa_s\widehat J_\nu \cdot \pa^\alpha D_\nu\,dx &  = -\frac1{2\nu^2} \frac{d}{ds} \|\widehat J_\nu\|_{H^{m-1}}^2 -\frac1{\nu^2} \frac{d}{ds} \sum_{|\alpha|\le m-1} \intt\pa^\alpha\widehat J_\nu \cdot\pa^\alpha A_\nu\,dx \cr
&\quad + \frac1{\nu^2} \sum_{|\alpha|\le m-1} \intt \pa^\alpha\widehat J_\nu \cdot \pa^\alpha\pa_sA_\nu\,dx .
\end{aligned}
\end{align}
Using $\widehat J_\nu=D_\nu-A_\nu$, we estimate
\begin{align*}
\frac1{\nu^2} \lt| \sum_{|\alpha|\le m-1} \intt \pa^\alpha\widehat J_\nu \cdot \pa^\alpha\pa_sA_\nu\,dx \rt| &\le \frac1{\nu^2} \|\widehat J_\nu\|_{H^{m-1}} \|\pa_sA_\nu\|_{H^{m-1}}
\cr
& \le \frac18\|D_\nu\|_{H^{m-1}}^2 + \frac{C}{\nu^2}\|A_\nu\|_{H^{m-1}}^2 + \frac{C}{\nu^2}\|\pa_sA_\nu\|_{H^{m-1}}^2.
\end{align*}
Similarly, we get
\[
\frac1{\nu^2} \lt| \sum_{|\alpha|\le m-1} \intt \pa^\alpha\calN_\nu \cdot \pa^\alpha D_\nu\,dx \rt| \le \frac18\|D_\nu\|_{H^{m-1}}^2 + \frac{C}{\nu^4} \|\calN_\nu\|_{H^{m-1}}^2,
\]
and
\[
\frac1{\nu^2} \lt| \sum_{|\alpha|\le m-1} \intt \pa^\alpha\Delta J_L^\nu \cdot \pa^\alpha D_\nu\,dx \rt| \le \frac18\|D_\nu\|_{H^{m-1}}^2 + \frac{C}{\nu^4} \|\Delta J_L^\nu\|_{H^{m-1}}^2 .
\]
Combining these estimates with \eqref{eq:defect-e0}--\eqref{eq:Jhat-t}, we deduce
\begin{align}\label{eq:defect-diff}
\begin{aligned}
&\frac12\|D_\nu\|_{H^{m-1}}^2 + \frac1{2\nu^2} \frac{d}{ds} \|\widehat J_\nu\|_{H^{m-1}}^2  + \frac1{\nu^2} \frac{d}{ds} \sum_{|\alpha|\le m-1} \intt \pa^\alpha\widehat J_\nu \cdot \pa^\alpha A_\nu\,dx \cr
&\quad \le \frac{C}{\nu^2} \|A_\nu\|_{H^{m-1}}^2 + \frac{C}{\nu^2} \|\pa_sA_\nu\|_{H^{m-1}}^2  + \frac{C}{\nu^4} \|\calN_\nu\|_{H^{m-1}}^2 + \frac{C}{\nu^4} \|\Delta J_L^\nu\|_{H^{m-1}}^2 .
\end{aligned}
\end{align} 
Integrating \eqref{eq:defect-diff} over $[0,T]$ and using $\widehat J_\nu(0)=0$,  we obtain
\begin{align}\label{eq:defect-int}
\begin{aligned}
& \int_0^T \|D_\nu(s)\|_{H^{m-1}}^2\,ds + \frac1{\nu^2} \|\widehat J_\nu(T)\|_{H^{m-1}}^2 \cr
&\quad \le \frac{C}{\nu^2} + \frac{C}{\nu^2} \|A_\nu(T)\|_{H^{m-1}}^2 + \frac{C}{\nu^2} \int_0^T \lt( \|A_\nu(s)\|_{H^{m-1}}^2 + \| \pa_sA_\nu(s)\|_{H^{m-1}}^2 \rt)ds \cr
&\quad \quad + \frac{C}{\nu^4} \int_0^T \|\calN_\nu(s)\|_{H^{m-1}}^2\,ds + \frac{C}{\nu^4} \int_0^T \|\Delta J_L^\nu(s)\|_{H^{m-1}}^2\,ds .
\end{aligned}
\end{align}
Here, we used the terminal cross-term estimate
\[
\frac1{\nu^2} \lt| \sum_{|\alpha|\le m-1} \intt \pa^\alpha\widehat J_\nu \cdot \pa^\alpha A_\nu\,dx \rt| \le \frac1{4\nu^2} \|\widehat J_\nu(T)\|_{H^{m-1}}^2 + \frac{C}{\nu^2} \|A_\nu(T)\|_{H^{m-1}}^2.
\]
By Lemmas \ref{lem:JL-nu} and \ref{lem:st-bds-nu}, the right-hand side of \eqref{eq:defect-int} is bounded by $C/\nu^2$, uniformly in $T$. Hence, we have
\[
\int_0^T \|D_\nu(s)\|_{H^{m-1}}^2\,ds + \frac1{\nu^2} \|\widehat J_\nu(T)\|_{H^{m-1}}^2 \le \frac{C}{\nu^2}.
\]
Since the estimate is uniform in $T$, taking the supremum over $T\ge0$ yields \eqref{eq:Dar-def-est}. This completes the proof.
\end{proof}

We now derive the convergence of the rescaled flux after subtracting the initial layer. Since
\[
J_\nu-J_L^\nu-\bar J = D_\nu - \lt[ \nabla(\rho_\nu-\bar\rho) + \rho_\nu\nabla(\Phi_\nu-\bar\Phi) + (\rho_\nu-\bar\rho)\nabla\bar\Phi \rt],
\]
the product estimate, elliptic regularity, and the density error estimates \eqref{eq:od-den-nu}--\eqref{eq:od-den-nu2} yield
\[
\int_0^\infty \|J_\nu(s)-J_L^\nu(s)-\bar J(s)\|_{H^{m-1}}^2\,ds \le \frac{C}{\nu^2}.
\]
Combining the density estimates \eqref{eq:od-den-nu}--\eqref{eq:od-den-nu2} with the flux estimate above, we obtain \eqref{eq:od-den-err} and \eqref{eq:od-flux-err}. This completes the proof of Theorem \ref{thm:main2}.

%
%
%
%
%
%

\section*{Acknowledgments}
The work of Y.-P. Choi was supported by NRF grant no. 2022R1A2C1002820 and RS-2024-00406821. 
 
%
%
%
%

%

%
%
%
%

\end{document}